\documentclass[final,onefignum,onetabnum]{siamart190516}

\usepackage[numbers]{natbib}
\usepackage{natbib,verbatim}
\usepackage{lipsum}
\usepackage{amsfonts}
\usepackage{graphicx}
\usepackage{subfigure}
\usepackage{epstopdf}
\usepackage{algorithmic}
\usepackage{amsmath,amssymb}
\ifpdf
  \DeclareGraphicsExtensions{.eps,.pdf,.png,.jpg}
\else
  \DeclareGraphicsExtensions{.eps}
\fi
\renewcommand\theequation{\thesection.\arabic{equation}}
\theoremstyle{plain}
\newtheorem{lem}{Lemma}[section]

	\newtheorem{thm}{Theorem}[section]

	\newtheorem{rem}[thm]{Remark}

    \newcommand{\nrm}[1]{\left\| #1 \right\|}

\newcommand{\bm}{\boldsymbol}

   \newcommand\dt {{\Delta t}}

\newcommand{\be}{\begin{equation}}
\newcommand{\ee}{\end{equation}}

\newcommand{\bse}{\begin{subequations}}
\newcommand{\ese}{\end{subequations}}
\def\benl{\begin{eqnarray*}}
\def\eenl{\end{eqnarray*}}

\def\be{\bm{e}}

\def\bx{\bm{x}}

\def\bmu1{\bm{\mu_1}}

\def\ln{{\rm ln}}

\newcommand{\ben}{\begin{eqnarray}}
\newcommand{\een}{\end{eqnarray}}
\newcommand{\beq}{\begin{equation}}
\newcommand{\eeq}{\end{equation}}
\newcommand{\bea}{\begin{array}}
\newcommand{\eea}{\end{array}}
\newcommand{\bef}{\begin{figure}(H)_h}
\newcommand{\eef}{\end{figure}}

\newsiamremark{hypothesis}{Hypothesis}
\crefname{hypothesis}{Hypothesis}{Hypotheses}
\newsiamthm{claim}{Claim}

\title{Unique solvability and error analysis of the Lagrange multiplier approach for gradient flows}

\author{
Qing Cheng \thanks{Department of Mathematics, Tongji University, Shanghai 200092, China (qingcheng$@$tongji.edu.cn). Key Laboratory of Intelligent Computing and Applications (Tongji University), Ministry of Education, China. }
\and		
Jie Shen
\thanks{Eastern Institute of Technology, Ningbo, Zhejiang 315200, P. R. China (jshen@eitech.edu.cn)}
\and
Cheng Wang\thanks{Department of Mathematics, University of Massachusetts, North Dartmouth, MA  02747, USA (cwang1@umassd.edu)}
}

\usepackage{amsopn}

\begin{document}
\graphicspath{ {Figures/} }
\maketitle

\begin{abstract}
The unique solvability and error analysis of the original  Lagrange multiplier approach proposed in \cite{ChengQ2020a} for gradient flows is studied in this paper. We identify a necessary and sufficient  condition that  must be satisfied for the nonlinear algebraic equation arising from  the original  Lagrange multiplier approach  to admit a unique solution in the neighborhood of its exact solution, and propose a modified Lagrange multiplier approach so that the computation can continue even if the aforementioned condition is not satisfied. Using Cahn-Hilliard equation as an example, we prove rigorously the unique solvability  and establish optimal error estimates of a second-order Lagrange multiplier scheme assuming this condition and that the time step is sufficient small. We also present numerical results to demonstrate that the  modified Lagrange multiplier approach is much more robust and can use much larger time step than the original  Lagrange multiplier approach.

\end{abstract}



\section{Introduction}

We consider the numerical approximation of a general gradient flow given by
	\begin{equation}\label{gflow}
		\partial_t \phi=-\mathcal{G}\frac{\delta E}{\delta \phi},
	\end{equation}
	where $E(\phi)=\frac 12(\mathcal{L}^{1/2}\phi,\mathcal{L}^{1/2}\phi)+(F(\phi),1)$, with $\mathcal{L}$ and $\mathcal{G}$ being positive definite operators in a suitable Hilbert space with inner product $(\cdot,\cdot)$, and $F(\phi)$ is a nonlinear potential function. An important property of \eqref{gflow} is an associated energy dissipation law:
	\begin{equation}\label{diss}
		\frac{d}{dt}E =-(\mathcal{G}\frac{\delta E}{\delta \phi}, \frac{\delta E}{\delta \phi}).
	\end{equation} 
	It  is highly desirable to design numerical schemes which can satisfy a discrete version of \eqref{diss}.
	
	In recent years, a great deal of efforts have been devoted to construct efficient and accurate energy dissipative schemes for various gradient flows in the form of  \eqref{gflow}; we refer to \cite{bao2004computing, du2008numerical, du2006simulating, MR4458898, MR4450101, shen19} and the references therein for  more details. For example, the convex splitting method~\cite{aristotelous13b, chen19b, Dong2021a, Dong2022a, dong19b, elliott93, eyre98, Yuan2021a}, which treats the convex part implicitly and the concave part explicitly, ensures the unique solvability and unconditional energy stability at a theoretical level. On the other hand, the price of this numerical approach is associated with a nonlinear solver at each time step, due to the fact that the nonlinear terms in the gradient flow usually correspond to a convex energy. Moreover, some higher order versions of this approach, in both the second and third accuracy orders, have been extensively studied~\cite{cheng2019a, cheng2022a, cheng16a, diegel16, guo16, guo21, yan18}, and the stability analysis for a modified energy functional, composed of the original free energy and a few numerical correction terms, has been reported. Again, a nonlinear solver has to be implemented in these higher order energy stable schemes, which has always been a huge numerical challenge. 
	
	To overcome the difficulty associated with a nonlinear solver in the numerical implementation, many linear approach efforts have been made for various gradient flows. In particular, the stabilization method is applied to the Cahn-Hilliard equation~\cite{LiD2017, LiD2017b, LiD2016a, shen10a}, in which an artificial regularization term is added to ensure the energy stability, either in terms of the original free energy, or a modified energy functional, usually under a global Lipshitz condition on the nonlinear part of the free energy. On the other hand, the IEQ approach proposed in \cite{YANG2017_3phase} gives a linear, and unconditionally energy stable (with respect to a modified energy) scheme, but it  requires solving a linear system with variable coefficients. The original scalar auxiliary variable (SAV) approach proposed in \cite{shen18a} leads to a linear, decoupled and unconditionally energy stable (with respect to a modified energy) scheme, which is very efficient and easy to implement, while it is not energy dissipative with respect to the original energy. On the other hand, the original Lagrange multiplier approach proposed in \cite{ChengQ2020b} leads to a linear, decoupled and unconditionally energy stable (with respect to the original energy) scheme, combined with a nonlinear algebraic equation for the Lagrange multiplier. In principle, this approach has essentially all the desired attributes for solving gradient flows; however, it is not clear whether the nonlinear algebraic equation admits a unique solution in the desired range.
	
	 Although there are ample numerical results indicating that the original Lagrange multiplier approach works well in many applications, but  there are cases where exceedingly small time steps are needed or  one is unable to find a suitable solution of this  nonlinear algebraic equation  \cite{MR4242420, ChengQ2020a}. Therefore, it is very important to identify condition(s) which can ensure the unique solvability, and modify the original Lagrange multiplier approach so that the computation can continue even if theses condition(s) are not satisfied.

We observe from  \eqref{gflow-SAVN-3}, which is the last equation in the   original Lagrange multiplier scheme \eqref{gflow-SAVN-1}-\eqref{gflow-SAVN-3}, that the Lagrange multiplier $\eta^{n+1/2}$ can not be uniquely determined if $(F(\phi^{n+1})-F(\phi^n),1)=0$. Therefore, it is necessary to make the following
assumption on  the exact solution: 
\begin{equation} 
 S_n:=	\frac{d}{dt}   \int_\Omega \, F (\Phi)  \, d \bx|_{t=t^n}\ne 0.
	\label{assumption-H-0} 
\end{equation} 
This assumption, of course, can not be satisfied  {\it a priori}  for the whole time interval.  Hence, it is necessary to modify the  Lagrange multiplier approach  so that the computation can continue even if \eqref{assumption-H-0}  is not satisfied at some time.

The main purpose of this work is to take the Cahn-Hilliard equation as an example of gradient flow to study the unique solvability of the nonlinear algebraic equation in the original Lagrange multiplier approach, and to carry out its error analysis. The main contributions of this work include:
\begin{itemize}
	\item We prove rigorously that if the assumption \eqref{assumption-H-0} is satisfied, then the original  Lagrange multiplier  scheme  \eqref{scheme-SAVN-1}-\eqref{scheme-SAVN-2} admits a unique solution in the interval $[1-\Delta t^{1/2}, 1-\Delta t^{1/2}]$.
	\item We propose a modified Lagrange multiplier approach to deal with the case when $|S_n|\ll 1$, and provide numerical results to show that the modified Lagrange multiplier approach is much more robust and can use much larger time steps than the original  Lagrange multiplier approach.
	\item Given a tolerance $\gamma$, and assume $|S_n|\ge \gamma$ in the time interval $[0,T]$, we  carry out an error analysis and establish optimal error estimates under the condition that the time step $\Delta t\le (\frac{\gamma}{4})^4$.
\end{itemize}
In fact, the unique solvability of the nonlinear algebraic equation in the  Lagrange multiplier approach turns out to be very challenging. First of all, it is observed that the implicit part of the numerical scheme \eqref{scheme-SAVN-1}-\eqref{scheme-SAVN-2} does not correspond to a globally monotone functional in terms of the numerical solution at the next time step. To overcome this difficulty, we have to apply certain local analysis technique to obtain the unique solvability, viewed as a perturbation of the exact solution at each time step. To achieve this goal, an {\it a-priori}  assumption has to be made at the previous time step, in terms of the convergence estimate. With the help of this {\it a-priori} assumption, the unique solvability can then be carefully proved under the  assumption  \eqref{assumption-H-0} on the exact solution. Subsequently, to recover the   {\it  priori}  assumption in the unique solvability analysis, we derive an optimal rate convergence analysis at the next time step. By a  mathematical induction argument, we are able to complete  the proof of unique solvability and error analysis. 

	The rest of this paper is organized as follows. In Section~\ref{sec:numerical scheme},  we recall a  second order  scheme for the gradient flows using the original Lagrange multiplier approach, and then present a modified   Lagrange multiplier approach to deal with the case when   \eqref{assumption-H-0} is not satisfied.  A numerical example is given to validate the efficiency of the improved Lagrange multiplier approach.  In Section~\ref{sec:solvability}, we consider Cahn-Hilliard equation as an example and establish the unique solvability of the nonlinear system of algebraic equations, under an {\it a-priori}  assumption on the previous time step. Afterward, an error analysis is presented in Section~\ref{sec:convergence}, and the {\it a-priori}  assumption is theoretically recovered. Finally, we provide some concluding remarks in Section~\ref{sec:conclusion}.

\section{The Lagrange multiplier approach and a modified version}  \label{sec:numerical scheme} 
We denote $\mu = \frac{\delta E}{\delta \phi}= \mathcal{L}\phi +  F'(\phi)$. A second order modified Crank-Nicolson scheme  has been proposed for the general gradient flow \eqref{gflow}, based on  the original Lagrange multiplier approach \cite{ChengQ2020a}:  
\begin{eqnarray} 
  && 
  \frac{\phi^{n+1} - \phi^n}{\dt} =-\mathcal{G}\mu^{n+1},   
   \label{gflow-SAVN-1}   
\\
  &&
    \mu^{n+1}
  =  \mathcal{L}(\frac 34\phi^{n+1}+\frac 14 \phi^{n-1}) + \eta^{n+1/2} F' (\phi^{*,n}),  
  \label{gflow-SAVN-2}   
  \\
  &&
  \Big( F (\phi^{n+1} ) - F (\phi^n) , 1 \Big) 
  = \eta^{n+1/2}  \Big( F' (\phi^{*,n}) ,  \phi^{n+1} - \phi^n \Big),
  \label{gflow-SAVN-3}   
\end{eqnarray} 
where $\phi^{*,n}=\frac32\phi^n-\frac12\phi^{n-1},\; \eta^{n+1/2}=\frac{\eta^{n+1}+\eta^n}{2}$. The energy  stability result for the above scheme \eqref{gflow-SAVN-1}-\eqref{gflow-SAVN-3} could be established, following a similar idea as in \cite{ChengQ2020a}; see the following theorem. 

 \begin{thm}  
 The numerical scheme \eqref{gflow-SAVN-1}-\eqref{gflow-SAVN-3} is unconditional stable and  satisfies the following energy dissipative law
 \begin{equation}
 \frac{E^{n+1}-E^n}{\delta t} = -(\mathcal{G}\mu^{n+1}, \mu^{n+1})-(\mathcal{L}(\phi^{n+1}-2\phi^n+\phi^{n-1}), \phi^{n+1}-2\phi^n+\phi^{n-1}),
 \end{equation}
 where the energy $E^{n+1}$ is defined as 
 \begin{equation}
 E^{n+1}= \frac 12 (\mathcal{L}\phi^{n+1},\phi^{n+1}) +\frac 18(\mathcal{L}(\phi^{n+1}-\phi^n), \phi^{n+1}-\phi^n) + (F(\phi^{n+1}),1).
 \end{equation}
 \end{thm}
 \begin{proof}
 Taking inner product of \eqref{gflow-SAVN-1}, \eqref{gflow-SAVN-2} with $\mu^{n+1}$, $\frac{\phi^{n+1} - \phi^n}{\dt}$ respectively, we obtain
 \begin{equation}\label{stable:1}
  ( \frac{\phi^{n+1} - \phi^n}{\dt}, \mu^{n+1} ) =-(\mathcal{G}\mu^{n+1}, \mu^{n+1}),
 \end{equation}
 and 
 \begin{equation}\label{stable:2}
  ( \frac{\phi^{n+1} - \phi^n}{\dt}, \mu^{n+1} ) = (\mathcal{L}(\frac 34\phi^{n+1}+\frac 14 \phi^{n-1}) + \eta^{n+1/2} F' (\phi^{*,n}), \frac{\phi^{n+1} - \phi^n}{\dt}). 
 \end{equation}
Notice the equality
\begin{equation}\label{in:e:stable}
\begin{split}
& (\mathcal{L}(\frac 34\phi^{n+1}+\frac 14 \phi^{n-1}), \phi^{n+1}-\phi^n)= \frac 18\{(\mathcal{L}(\phi^{n+1}-\phi^n), \phi^{n+1}-\phi^n)- (\mathcal{L}(\phi^{n}-\phi^{n-1}), \phi^{n}-\phi^{n-1}) \\&+ (\mathcal{L}(\phi^{n+1}-2\phi^n+\phi^{n-1}), \phi^{n+1}-2\phi^n+\phi^{n-1})\}+\frac 12\{ (\mathcal{L}\phi^{n+1},\phi^{n+1})- (\mathcal{L}\phi^{n},\phi^{n})\}.
 \end{split}
\end{equation} 
 Combining  \eqref{stable:1} and \eqref{stable:2}   and using \eqref{gflow-SAVN-3} and \eqref{in:e:stable}, we derived the desired  energy dissipative law. 
 \end{proof}

	 The implementation process of the above scheme can be outlined as follows. First, we define
\begin{equation}\label{phimu}
	 \phi^{n+\frac 12}=\phi_1^{n+\frac 12}+\eta^{n+\frac 12}\phi_2^{n+\frac 12},\quad \mu^{n+\frac 12}=\mu_1^{n+\frac 12}+\eta^{n+\frac 12}\mu_2^{n+\frac 12}.
\end{equation} 
Substituting $(\phi^{n+\frac 12},\mu^{n+\frac 12})$ into  \eqref{gflow-SAVN-1}-\eqref{gflow-SAVN-3}, we are able to obtain $\phi_1^{n+\frac 12}$ and $\phi_2^{n+\frac 12}$ from the following two linear systems
\begin{equation}\label{g:1}
\frac{\phi_1^{n+\frac 12} - \phi^n}{\dt} =-\mathcal{G}\mu_1^{n+\frac 12}, \quad 
\mu_1^{n+\frac 12} =  \mathcal{L}(\frac 34\phi_1^{n+1} +\frac 14 \phi^{n-1}),
\end{equation}
and 
\begin{equation}\label{g:2}
\frac{\phi_2^{n+\frac 12}}{\dt} =-\mathcal{G}\mu_2^{n+\frac 12},\quad
\mu_2^{n+\frac 12} =  \mathcal{L}(\frac 34\phi_2^{n+1})  + F' (\phi^{*,n}). 
\end{equation}
It is obvious that $\phi_1^{n+\frac 12}$ and $\phi_2^{n+\frac 12}$ can be solved uniquely from \eqref{g:1}-\eqref{g:2}. Subsequently, a substitution of \eqref{phimu} into \eqref{gflow-SAVN-3} leads to the following 
 nonlinear algebraic equation for $\eta^{n+1/2}$:
\begin{equation}\label{g:nonlinear-2}
g_n(\eta):= \int_\Omega \, \Big( F (\phi_1^{n+\frac 12} + \eta  \phi_2^{n+\frac 12}) - F (\phi^n)  \Big) \, d \bx 
  - \eta  \int_\Omega \, F' (\phi^{*,n}) (\phi_1^{n+\frac 12} + \eta  \phi_2^{n+\frac 12} - \phi^n) \, d \bx = 0.  
\end{equation}
As mentioned in the introduction, \eqref{g:nonlinear-2}  may not be uniquely solvable near $\eta=1$ if the assumption \eqref{assumption-H-0} is not satisfied. Hence, we need to modify the  Lagrange multiplier approach so that the computation can continue even if \eqref{assumption-H-0} is not satisfied. 

Below we introduce a modified Lagrange multiplier approach for gradient flow \eqref{gflow}. For a given  tolerance $\gamma\ll 1$, we proceed as follows.

We first compute  $\phi_1^{n+\frac 12}$ and $\phi_2^{n+\frac 12}$ from  \eqref{g:1} and \eqref{g:2}, respectively. Then, we set $\tilde \phi^{n+1}=\phi_1^{n+\frac 12} +   \phi_2^{n+\frac 12}$. If $e_{n+1}= |\frac{\int_{\Omega}F(\tilde \phi^{n+1})-F(\phi^n)dx}{\Delta t}|>\gamma$, we continue with the original Lagrange multiplier approach; otherwise 
we set $\eta^{n+1/2}=1$  and  $\phi^{n+1}=\tilde \phi^{n+1}$. 

More precisely, the modified Lagrange multiplier algorithm is outlined below. 

\medskip
\noindent{\bf Modified Lagrange multiplier approach:}\\
\rule[4pt]{14.3cm}{0.05em}\\
\textbf{Given}  Solutions at time steps $n$ and $n-1$; parameters $\gamma$, and the preassigned  time step $\Delta t$.
\begin{description}
	\item[Step 1] Compute  $\phi_1^{n+\frac 12}$ and $\phi_2^{n+\frac 12}$ from  \eqref{g:1} and \eqref{g:2}, and set $\tilde \phi^{n+1}=\phi_1^{n+\frac 12} +   \phi_2^{n+\frac 12}$.
	\item[Step 2] Calculate $e_{n+1}= \frac{\int_{\Omega}F(\tilde\phi^{n+1})-F(\phi^n)dx}{\Delta t} $.
	\item[Step 3] \textbf{if} $e_{n+1}>\gamma$, \textbf{then}
	\item[Step 4] Determine $\eta^{n+1/2}$ from \eqref{g:nonlinear-2}.
	\item[Step 5] \textbf{else}\\
Set $\phi^{n+1}=\tilde \phi^{n+1}$.
	\item[Step 6] \textbf{endif}
\end{description}
\rule[12pt]{14.3cm}{0.05em}\\

We now provide a numerical example to demonstrate the effectiveness of this modified approach. We consider the Cahn-Hilliard equation 
\begin{equation}
\phi_t  +  \varepsilon^2\Delta^2\phi -\Delta ((\phi^2-1)\phi)= 0 , 
\end{equation}
with the initial condition 
\begin{equation}\label{ini:3d}
	u(t=0)=0.3+0.01\,{\rm rand}(x,y), \quad 
	\mbox{${\rm rand}(x,y)$: uniform random distribution in $[-1,1]$} . 
\end{equation}
The interface width parameter is set to be $\varepsilon^2=0.06$. 

In Figure~\ref{Fig:1}(a-c), we plot the evolution of  energy, iterations and $\eta$ with respect to time by using the modified Crank-Nicolson scheme based on the original Lagrange multiplier approach. We observe that even with a very small time step $\Delta t=10^{-7}$, the scheme based on the original Lagrange multiplier approach failed to converge at about $t=1.65\times 10^{-3}$.  In Figure~\ref{Fig:1}(d-f), we plot the evolution of  energy, iterations and $\eta$ with respect to time by using the  scheme based on the modified Lagrange multiplier approach with $\gamma=\Delta t$ and $\Delta t=10^{-3}$. For the sake of comparison, we also plot the energy evolution by using a second-order SAV scheme with $\Delta t=10^{-5}$. We observe from Figure~\ref{Fig:1}(d) that the energy curves obtained by both methods overlap and decrease monotonically, indicating that the modified Lagrange multiplier approach leads to correct results ever at a relatively large time step  $\Delta t=10^{-3}$. We also observe from Figure~\ref{Fig:1}(e) that the modified approach is activated (i.e., $e_{n+1}\le \gamma$) in a large time interval, while only one iteration is needed for solving the nonlinear algebraic equation \eqref{g:nonlinear-2} when $e_{n+1}\ge \gamma$. These results indicate that the modified Lagrange multiplier approach is very effective.

	 \begin{figure}[htbp]
	\centering
	\subfigure[Energy with $\Delta t=10^{-7}$ ]{
	\includegraphics[width=0.25\textwidth,clip==]{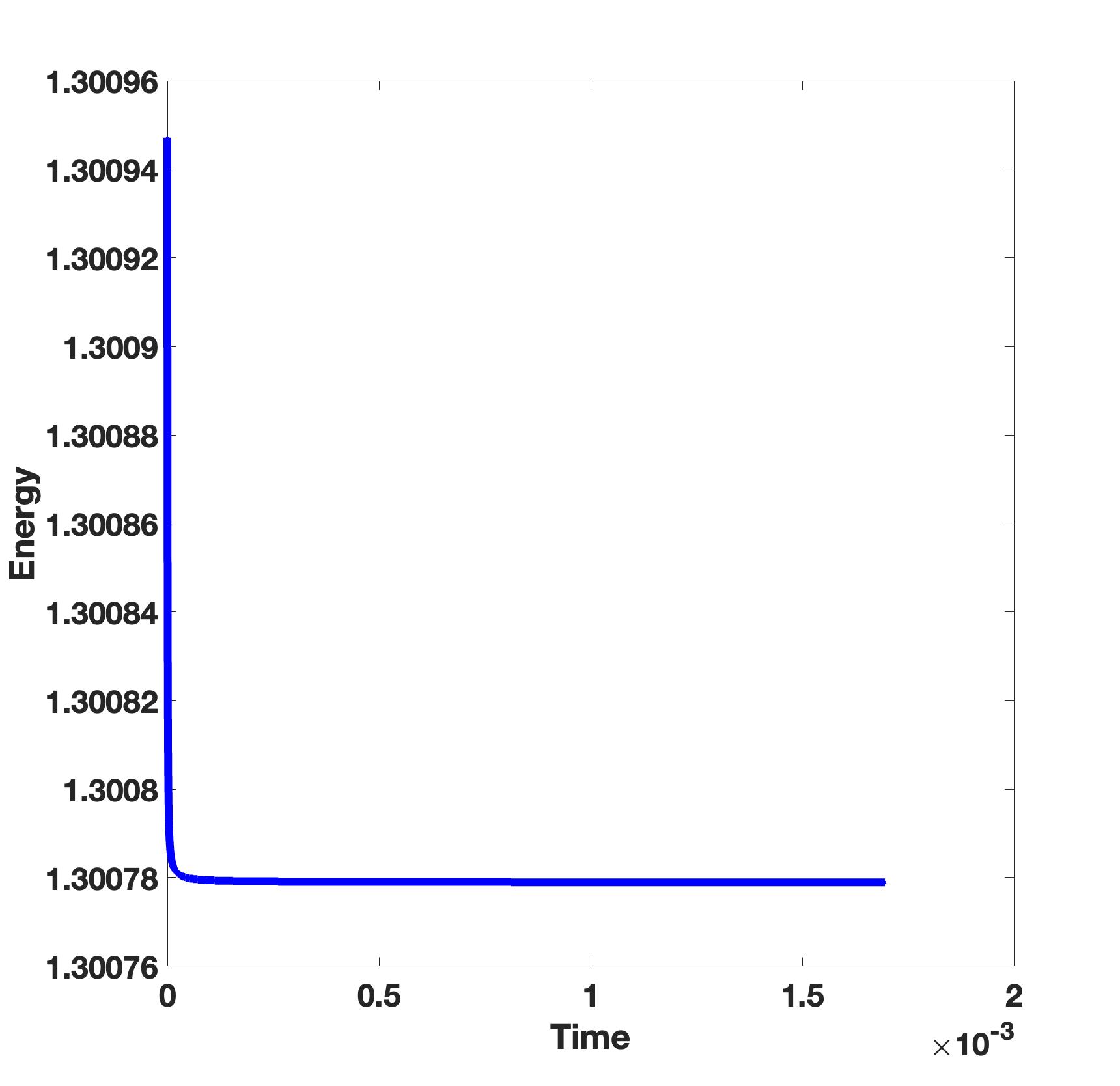}}
	\subfigure[Iterations with  $\Delta t=10^{-7}$]{
	\includegraphics[width=0.25\textwidth,clip==]{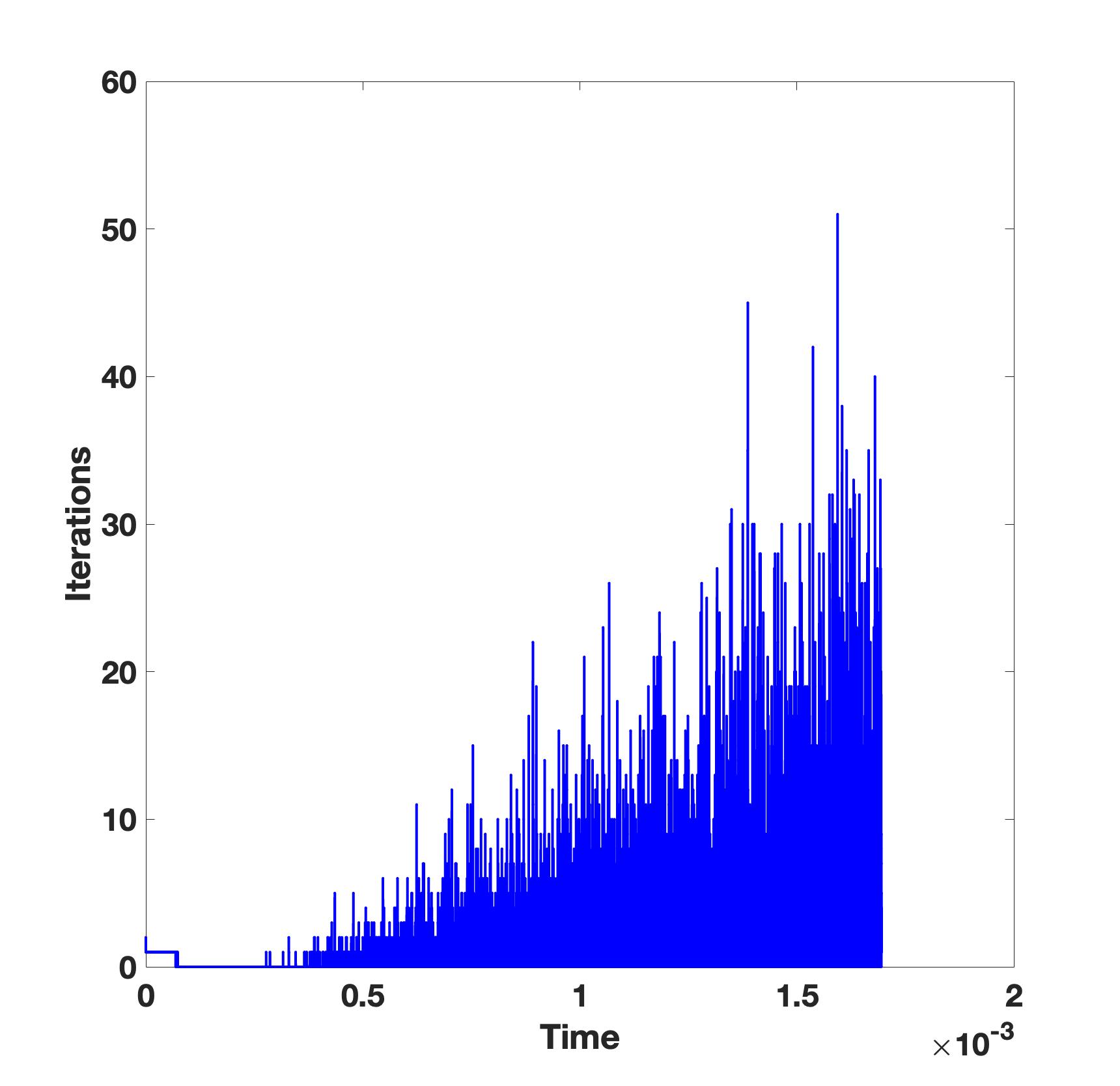}}
	\subfigure[$\eta$ with $\Delta t=10^{-7}$]{
	\includegraphics[width=0.25\textwidth,clip==]{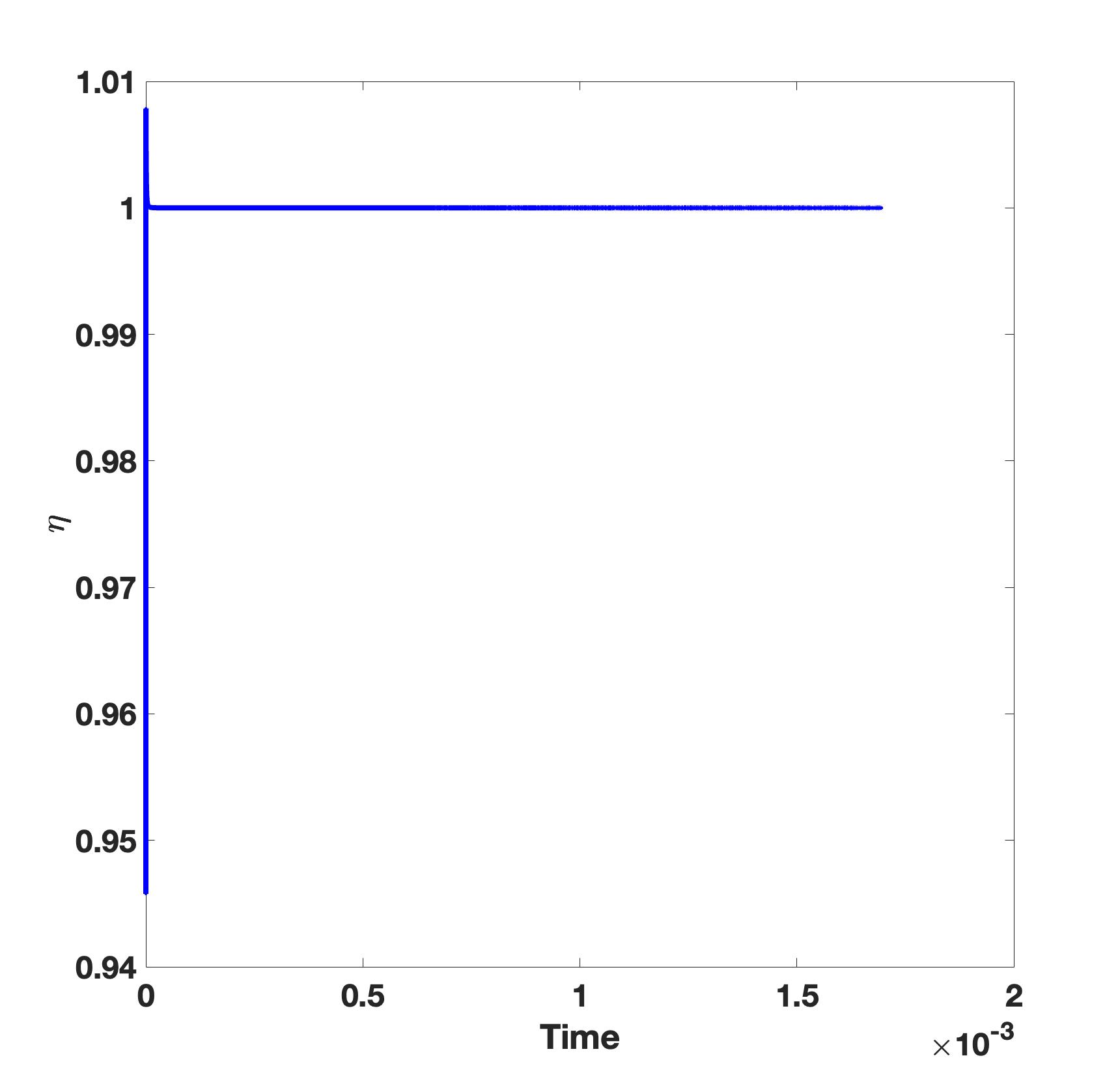}}
	\subfigure[Energy with  $\Delta t=10^{-3}$]{
	\includegraphics[width=0.25\textwidth,clip==]{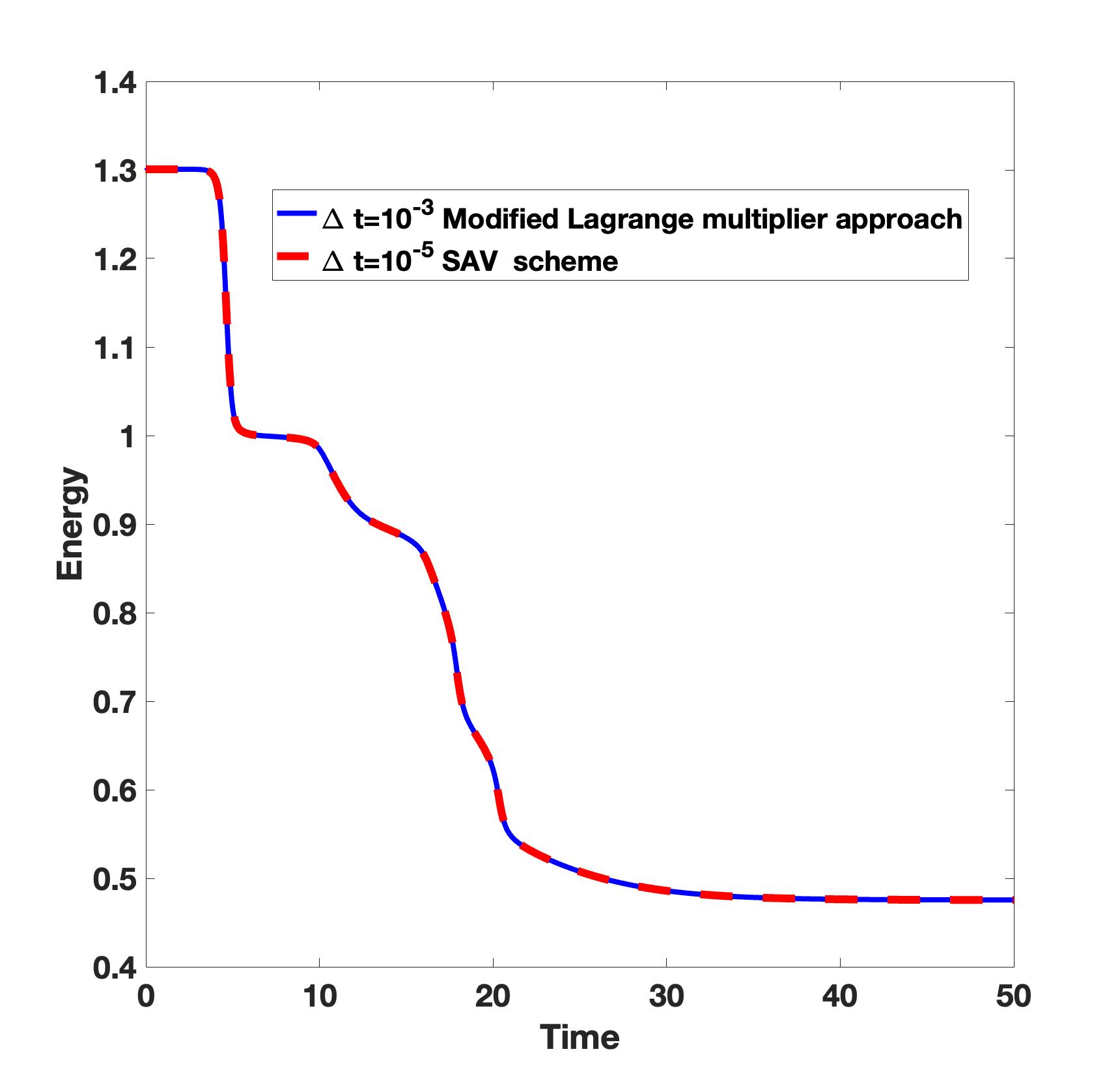}}
	\subfigure[Iterations with  $\Delta t=10^{-3}$]{
	\includegraphics[width=0.25\textwidth,clip==]{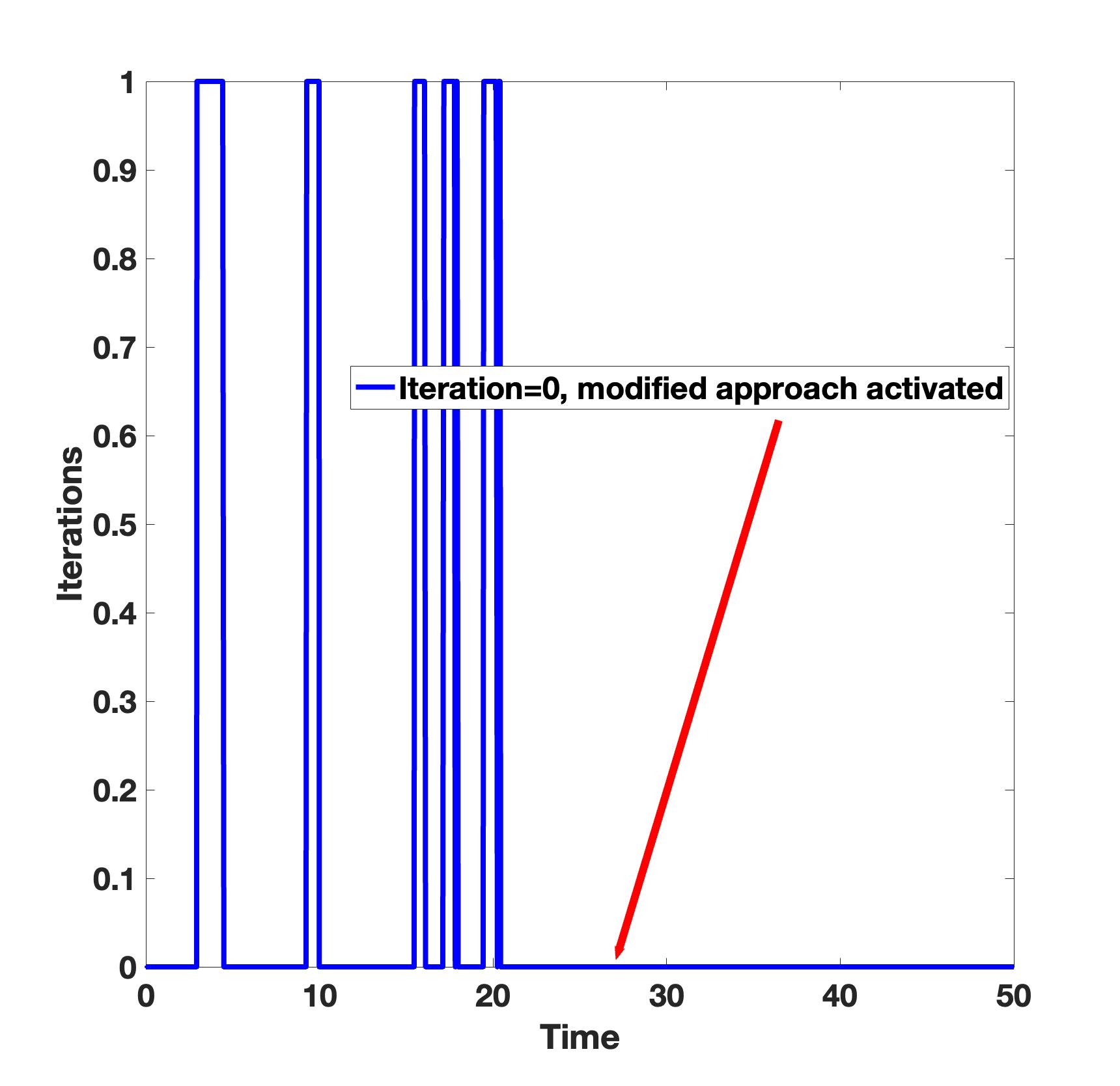}}
	\subfigure[$\eta$ with $\Delta t=10^{-3}$]{
	\includegraphics[width=0.25\textwidth,clip==]{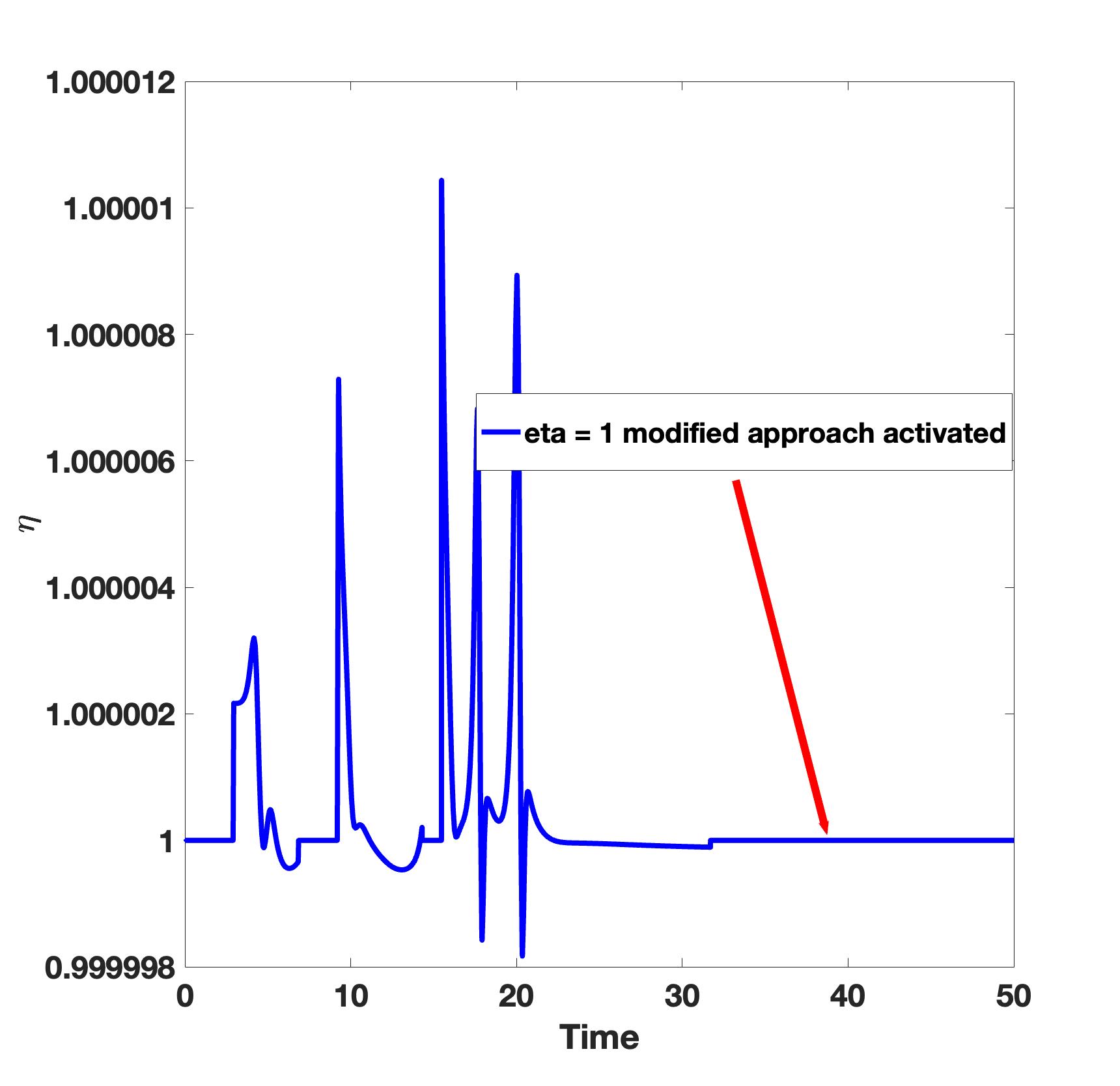}}
	\caption{ (a-c): The evolution of energy, iterations and $\eta$  using the original Lagrange multiplier approach with $\Delta t=10^{-7}$. (d-f): The evolution of energy, iterations and $\eta$  using modified Lagrange multiplier approach with $\Delta t=10^{-3}$. } \label{Fig:1}
\end{figure}

\section{The unique solvability of  \eqref{nonlinear-2}} \label{sec:solvability}
  
  In this section, we provide the unique solvability analysis of  \eqref{nonlinear-2}. To fix the idea, we  consider the Cahn-Hilliard equation, a typical  gradient flow. In this physical model, the energy functional is given by
  \begin{equation} 
  	E (\phi) := \int_\Omega \left( \frac14 \phi^4 
  	- \frac12 \phi^2 + \frac14 + \frac{\varepsilon^2}{2} |\nabla \phi|^2  
  	\right) \,\mathrm{d}\bx \ , 
  	\label{energy-CH-1} 
  \end{equation} 
where  the constant $\varepsilon >0$ stands for the interface width parameter, and $\mathcal{G}=-\Delta$. In turn, the Cahn-Hilliard equation can be written as
\begin{equation} 
	\partial_t \phi = \Delta \mu = \Delta 
	\left( \phi^3 - \phi - \varepsilon^2 \Delta \phi \right),
	\label{equation-CH}
\end{equation}  
in which $\mu$ is the chemical potential 
  \begin{equation} 
  	\mu: = \frac{\delta E}{\delta \phi} = \phi^3 - \phi - \varepsilon^2 \Delta \phi.
  	\label{chem-pot-CH}
  \end{equation} 
For simplicity, periodic boundary conditions are imposed for both the phase variable, $\phi$, and the chemical potential, $\mu$. An extension to the case of homogeneous Neumann boundary condition would be straightforward. 

In more details, the scheme \eqref{gflow-SAVN-1}-\eqref{gflow-SAVN-3} could be represented as 
\begin{eqnarray} 
  && 
  \frac{\phi^{n+1} - \phi^n}{\dt} = \Delta \Bigl( F' (\phi^{*,n}) \eta^{n+1/2} 
  - \varepsilon^2  \Delta ( \frac34 \phi^{n+1} + \frac14 \phi^{n-1} ) \Bigr) ,   
   \label{scheme-SAVN-1}   
\\
  &&
  \Big( F (\phi^{n+1} ) - F (\phi^n) , 1 \Big) 
  = \eta^{n+1/2}  \Big( F' (\phi^{*,n}) ,  \phi^{n+1} - \phi^n \Big). 
  \label{scheme-SAVN-2}   
\end{eqnarray} 
Meanwhile, the Lagrange multiplier $\eta^{n+1/2}$ is a solution of the nonlinear algebraic equation
\begin{eqnarray} 
\hskip 1cm
  g_n (\eta): = \int_\Omega \, \Big( F (p^n + \eta \dt q^n) - F (\phi^n)  \Big) \, d \bx 
  - \eta  \int_\Omega \, F' (\phi^{*,n}) (p^n + \eta \dt q^n - \phi^n) \, d \bx = 0.  
  \label{nonlinear-2}   
\end{eqnarray} 
Setting $	\phi^{n+1} = p^n + \eta^{n+1/2} \dt q^n$ and plugging it into \eqref{scheme-SAVN-1}, we find that 
\begin{eqnarray} 
  &&  
	p^n = ( I + \frac34 \varepsilon^2 \dt \Delta^2 )^{-1}  
	( \phi^n - \frac14 \varepsilon^2 \dt \Delta^2  \phi^{n-1} ) ,  \label{nonlinear-1-1}   
	\\
	&& 
	q^n = ( I + \frac34 \varepsilon^2 \dt \Delta^2 )^{-1}  \Delta F' (\phi^{*,n})  .    
	\label{nonlinear-1-2} 
\end{eqnarray}

\subsection{The main result}
We aim to provide a theoretical analysis for the nonlinear algebraic equation \eqref{nonlinear-2}, by making use of certain localized estimates.     
The numerical error function is defined as 
\begin{equation} 
	e^k := \Phi^k - \phi^k ,  \quad \forall k \ge 0 ,  \label{error function-1} 
\end{equation} 
in which $\Phi^k$ is the exact solution to the original PDE (\ref{equation-CH}), and $\phi^k$ is denoted as the numerical solution of (\ref{scheme-SAVN-1})-(\ref{scheme-SAVN-2}). With an initial data with sufficient regularity, it is assumed that the exact solution $\Phi$ has regularity of class $\mathcal{R}$: 
\begin{equation}
	\Phi \in \mathcal{R} := C^3 (0,T; C^2) \cap C^2 (0,T; C^6) \cap L^\infty (0,T; H^8).
	\label{assumption:regularity.1}
\end{equation}
In turn, the following functional bounds are available for the exact solution: 
\begin{equation} 
	\| \partial_t^3 \Phi \|_{L^\infty (0,T^*; C^2) } 
	+ \| \Phi_{tt} \|_{L^\infty (0,T^*;  C^6) }  
	+ \| \Phi_t \|_{L^\infty (0,T^*;  C^6) }  \le A^* ,  \, \, \, 
	\| \Phi^k \|_{H^8} \le A^* , \, \, \forall k \ge 0 .  
	\label{exact-bound-1}
\end{equation}

To proceed with the nonlinear analysis for (\ref{nonlinear-2}), we begin with the following  {\it a priori}  assumption for the numerical error at the previous time steps: 
\begin{equation} 
	\| e^k \|_{H^4} \le \dt^{3/2} ,   \quad \| e^k \|_{H^6} \le \dt ,  \quad 
 \forall k \le n .  \label{a priori-1} 
\end{equation} 
In turn, the following $H^4$ and $H^6$ bounds are valid for the numerical solution at the previous time steps: 
\begin{eqnarray} 
	\| \phi^k \|_{H^6} \le \| \Phi^k \|_{H^6} + \| e^k \|_{H^4} \le A^* + \dt  
	\le A^* +1 := \tilde{A}_1,  \quad 
	\mbox{for} \, \, \, k= n, n-1 .  \label{a priori-2} 
\end{eqnarray} 
In particular, by the fact that $\phi^{*,n} = \frac32 \phi^n - \frac12 \phi^{n-1}$, we have the following   {\it a-priori}  bounds for $\phi^{*,n}$: 
\begin{eqnarray} 
	\| \phi^{*,n}  \|_{H^2} , \, \,  \| \phi^{*,n}  \|_{H^4} , \, \, \| \phi^{*,n} \|_{H^6} 
	\le 2 \tilde{A}_1 ,  \quad  
	\| \phi^{*,n}  \|_{L^\infty} \le C \tilde{A}_1  . 
	\label{a priori-3}  
\end{eqnarray} 

The  {\it a priori}  assumption (\ref{a priori-1}) will be recovered by the convergence estimate at the next time step, as will be demonstrated in the later analysis. 

The following theorem is the main result of this section. 

\begin{thm} \label{thm: solvability} 
  Suppose the exact solution $\Phi$ for the Cahn-Hilliard equation (\ref{equation-CH}), of regularity class $\mathcal{R}$, satisfies the  assumption (\ref{assumption-H-0}) with $S_n \ge 4\dt^{\frac{1}{4}}$. We also make the {\it a-priori}  assumption (\ref{a priori-1}). Then the nonlinear scalar equation (\ref{nonlinear-2})  has a unique solution in  $[1 - S_n^2/16, 1 + S_n^2/16]$. 
\end{thm}   

We note that the numerical results presented in the last Section indicate that  the condition  $S_n \ge 4\dt^{\frac{1}{4}}$ (i.e. $\Delta t\le (S_n/4)^4$)  is most likely too pessimistic.

A few preliminary estimates are needed before we can proceed with  the proof of this theorem. 

\subsection{Some preliminary estimates }
For the sake of simplicity, we remove the dependent on $n$ from all constants ${A_k}$ below.

\begin{lem}  \label{lem 1} 
  Given $\phi^n$, $\phi^{n-1}$. Under the  assumption (\ref{assumption-H-0}) and the   {\it a-priori}  assumption (\ref{a priori-1}), we have the following estimates: 
\begin{eqnarray} 
  && 
    \| p^n - \phi^n \|_{L^\infty} \le A_1 \dt ,    \quad 
    \| p^n \|_{L^\infty} \le A_2 ,  \quad  \| q^n \|_{L^\infty} \le A_3 ,  \label{lem 1-0-1}  
\\
  && 
  \|  \phi^{*,n}  - \phi^n \|_{L^\infty} \le A_4 \dt ,  \quad  
  \|  \phi^{*,n}  \|_{L^\infty} \le A_5 ,   \quad 
  \|  \phi^{*,n}  - p^n \|_{L^\infty} \le A_6 \dt ,  \label{lem 1-0-2}  
\\
  && 
  \| \Delta \Big( F' ( p^n) - F' (\phi^n) \Bigr) \| \le A_7 \dt ,   \quad 
  \| \Delta \Big( F' ( \phi^{*,n}) - F' (\phi^n) \Bigr) \| \le A_8 \dt ,  \label{lem 1-0-3} 
\\
  &&   
  \| F' ( \phi^{*,n}) - F' (\phi^n) \|_{L^\infty} \le A_9 \dt , \quad 
  \| F' ( \Phi^n) - F' (\phi^n) \|_{L^\infty} \le A_{10} \dt , 
  \label{lem 1-0-4} 
\\
  && 
  \| \Delta \Big( F' ( \Phi^n) - F' (\phi^n) \Bigr) \| \le A_{11} \dt ,  \label{lem 1-0-5} 
\end{eqnarray} 
in which $A_j$ ($1 \le j \le 10$) only depends on $A^*$.    
\end{lem}

\begin{proof} 
By the representation formula (\ref{nonlinear-1-1}) for $p^n$, we see that 
\begin{eqnarray} 
  && 
   p^n - \phi^n = - \varepsilon^2 \dt ( I + \frac34 \varepsilon^2 \dt \Delta^2 )^{-1}  \Delta^2 ( \frac34 \phi^n + \frac14  \phi^{n-1} )  ,  
   \label{lem 1-1-1-a} 
\\
  && 
   \Delta ( p^n - \phi^n ) = - \varepsilon^2 \dt 
   ( I + \frac34 \varepsilon^2 \dt \Delta^2 )^{-1}  \Delta^3 ( \frac34 \phi^n + \frac14  \phi^{n-1} )  . 
   \label{lem 1-1-1-b}   
\end{eqnarray} 
Meanwhile, by the   {\it a priori}  $H^6$ estimate (\ref{a priori-2}) for the numerical solution $\phi^n$, $\phi^{n-1}$, we have 
\begin{eqnarray} 
   \| \Delta^2 ( \frac34 \phi^n + \frac14  \phi^{n-1} ) \| ,  \,  \,  
   \|  \Delta^3 ( \frac34 \phi^n + \frac14  \phi^{n-1} )  \| 
   \le  (\frac34 + \frac14 ) \tilde{A}_1 = \tilde{A}_ 1. 
   \label{lem 1-1-2}   
\end{eqnarray}  
On the other hand, the following inequality is always available, because all the eigenvalues associated with the operator $( I + \frac34 \varepsilon^2 \dt \Delta^2 )^{-1}$ is bounded by 1: 
\begin{eqnarray} 
  \| ( I + \frac34 \varepsilon^2 \dt \Delta^2 )^{-1} f  \|  \le \| f \| ,   \quad 
  \forall f \in L^2 (\Omega) . 
  \label{lem 1-1-3}   
\end{eqnarray}   
Then we arrive at 
\begin{eqnarray} 
  && 
   \| p^n - \phi^n \| \le \varepsilon^2 \dt 
  \|  \Delta^2 ( \frac34 \phi^n + \frac14  \phi^{n-1} ) \| 
  \le  \tilde{A}_1 \varepsilon^2  \dt ,  
    \label{lem 1-1-4-a} 
\\
  && 
   \| \Delta ( p^n - \phi^n ) \| \le  \varepsilon^2 \dt 
   \| \Delta^3 ( \frac34 \phi^n + \frac14  \phi^{n-1} )  \|  
    \le  \tilde{A}_1 \varepsilon^2  \dt .    
    \label{lem 1-1-4-b}   
\end{eqnarray} 
Therefore, an application of 3-D Sobolev embedding implies that 
\begin{eqnarray} 
  \| p^n - \phi^n \|_{L^\infty} \le C \| p^n - \phi^n \|_{H^2}  
  \le C ( \| p^n - \phi^n \| +   \| \Delta ( p^n - \phi^n ) \| ) 
  \le  C  \tilde{A}_1 \varepsilon^2  \dt ,   \label{lem 1-1-4-c}   
\end{eqnarray} 
with the elliptic regularity used in the second step. In turn, the first inequality of (\ref{lem 1-0-1}) has been proved by taking $A_1 = C \tilde{A}_1 \varepsilon^2$. 

The proof of the second inequality in (\ref{lem 1-0-1}) is more straightforward: 
\begin{eqnarray*} 
    \| p^n \|_{L^\infty}  \le \| \phi^n \|_{L^\infty} +  \| p^n - \phi^n \|_{L^\infty} 
    \le C \| \phi^n \|_{H^2} +  \| p^n - \phi^n \|_{L^\infty}  
    \le C \tilde{A}_1 + A_1 \dt \le  C \tilde{A}_1 + 1 , 
     \label{lem 1-1-5}   
\end{eqnarray*} 
provided that $A_1 \dt \le 1$. Again, the elliptic regularity has been recalled in the second step. As a result, the proof for second inequality in (\ref{lem 1-0-1}) is complete, by taking $A_2 = C \tilde{A}_1 + 1$. 

For the third inequality in (\ref{lem 1-0-1}), we begin with the following expansion, based on the fact that $F' (\phi) = \phi^3 - \phi$: 
\begin{eqnarray} 
   \Delta F' (\phi^{*,n}) = ( 3 (\phi^{*,n})^2 -1 ) \Delta (\phi^{*,n})  
   + 6 \phi^{*,n} | \nabla (\phi^{*,n}) |^2 .   \label{lem 1-1-6}   
\end{eqnarray}      
In turn, a combination of H\"older inequality and Sobolev embedding implies that 
\begin{eqnarray} 
   \| \Delta F' (\phi^{*,n}) \| &\le& ( 3 \| \phi^{*,n} \|_{L^\infty}^2 +1 ) \| \Delta (\phi^{*,n}) \|  
   + 6 \| \phi^{*,n} \|_{L^\infty} \cdot \| \nabla (\phi^{*,n}) \|_{L^4}^2  \nonumber 
\\
  &\le& 
    C ( \| \phi^{*,n} \|_{H^2}^2 +1 ) \| \Delta (\phi^{*,n}) \|  
   + C \| \phi^{*,n} \|_{H^2}^3 
   \le  C ( \| \phi^{*,n}  \|_{H^2}^3  +  \| \phi^{*,n} \|_{H^2} ) . 
     \label{lem 1-1-7}   
\end{eqnarray}   
A similar estimate could also be derived; the details are skipped for the sake of brevity: 
\begin{eqnarray} 
   \| \Delta^2 F' (\phi^{*,n}) \| 
   \le  C ( \| \phi^{*,n}  \|_{H^4}^3  +  \| \phi^{*,n} \|_{H^4} ) . 
     \label{lem 1-1-8}   
\end{eqnarray}   
As a consequence, we arrive at 
\begin{eqnarray*} 
  &&
  \| q^n \|  = \| ( I + \frac34 \varepsilon^2 \dt \Delta^2 )^{-1}  \Delta F' (\phi^{*,n}) \| 
  \le  \| \Delta F' (\phi^{*,n}) \|  
  \le  C ( \| \phi^{*,n}  \|_{H^2}^3  +  \| \phi^{*,n} \|_{H^2} )  ,     \label{lem 1-1-9-a}     
\\
  &&
  \| \Delta q^n \|  = \| ( I + \frac34 \varepsilon^2 \dt \Delta^2 )^{-1} 
   \Delta^2 F' (\phi^{*,n}) \| 
  \le  \| \Delta^2 F' (\phi^{*,n}) \|  
 \nonumber\\&& \le  C ( \| \phi^{*,n}  \|_{H^4}^3  +  \| \phi^{*,n} \|_{H^4} )  .   \label{lem 1-1-9-b}  
\end{eqnarray*} 
Furthermore, its combination with the   {\it a priori}  bound (\ref{a priori-3}) (for $\phi^{*,n}$) results in 
\begin{eqnarray*} 
  \| q^n \|  , \, \, \| \Delta q^n \|  \le  C (  \tilde{A}_1^3  +   \tilde{A}_1 )  ,    
    \label{lem 1-1-10}  
\end{eqnarray*} 
and an application of Sobolev embedding yields 
\begin{eqnarray*} 
  \| q^n \|_{L^\infty} \le C \| q^n \|_{H^2}  
  \le C ( \| q^n \|  + \| \Delta q^n \|  ) 
  \le  C (  \tilde{A}_1^3  +   \tilde{A}_1 )  . \label{lem 1-1-11}  
\end{eqnarray*} 
This finishes the proof of the third inequality in (\ref{lem 1-0-1}), by taking $A_3 =  C (  \tilde{A}_1^3  +   \tilde{A}_1 )$. 

  The rest inequalities in (\ref{lem 1-0-1})-(\ref{lem 1-0-5}) could be analyzed in the same fashion. The  details are skipped for simplicity of presentation. 
\end{proof} 

We aim to prove that the nonlinear equation (\ref{nonlinear-2}) has a unique solution in a neighborhood of $1$. An estimate of the value for $g_n (1)$ is given by the following lemma. 

\begin{lem}  \label{lem 2}
  Given $\phi^n$, $\phi^{n-1}$, under the  assumption (\ref{assumption-H-0}) and the   {\it a priori}  assumption (\ref{a priori-1}), we have $ | g_n (1) | \le A_{12} \dt^2$, with $A_{12}$ only depending on $A^*$.    
\end{lem}

\begin{proof} 
  We begin with the expansion of $g_n(1)$: 
\begin{eqnarray} 
  g_n (1) = \int_\Omega \, \Big( F (p^n + \dt q^n) - F (\phi^n)   
  - F' (\phi^{*,n}) (p^n + \dt q^n - \phi^n) \Big) \, d \bx .  
  \label{lem 2-1}   
\end{eqnarray} 
An application of intermediate value theorem implies that 
\begin{eqnarray*} 
   F (p^n + \dt q^n) - F (\phi^n)  = F' (\xi^{(1)})  (p^n + \dt q^n - \phi^n) , \quad 
   \mbox{with $\xi^{(1)}$ between $p^n + \dt q^n$ and $\phi^n$} . 
    \label{lem 2-2}   
\end{eqnarray*} 
As a consequence, we get 
\begin{eqnarray} 
  &&
  F (p^n + \dt q^n) - F (\phi^n)  - F' (\phi^{*,n}) (p^n + \dt q^n - \phi^n)   \nonumber 
\\ 
  &=&  
    \Big( F' (\xi^{(1)}) - F' (\phi^{*,n}) \Big) (p^n + \dt q^n - \phi^n)   \nonumber 
\\
  &=&  
    F'' (\xi^{(2)}) ( \xi^{(1)} - \phi^{*,n})  (p^n + \dt q^n - \phi^n)   , \quad 
   \mbox{with $\xi^{(2)}$ between $\xi^{(1)}$ and $\phi^{*,n}$} . 
    \label{lem 2-3}   
\end{eqnarray} 
By the inequality (\ref{lem 1-0-1}), the following estimate is available: 
\begin{eqnarray} 
  \| p^n + \dt q^n - \phi^n \|_{L^\infty} 
  \le  \| p^n - \phi^n \|_{L^\infty}  +  \dt \| q^n \|_{L^\infty}  
  \le (A_1 + A_3) \dt .  \label{lem 2-4}     
\end{eqnarray}  
Moreover, since $\xi^{(1)}$ is between $p^n + \dt q^n$ and $\phi^n$, we see that 
\begin{eqnarray} 
  \| \xi^{(1)} - \phi^{*,n} \|_{L^\infty} 
  \le  \max \Big( \|  p^n + \dt q^n - \phi^{*,n} \|_{L^\infty} ,  
   \| \phi^n - \phi^{*,n} \|_{L^\infty} \Big) .  \label{lem 2-5}     
\end{eqnarray}  
Meanwhile, the following inequalities are available: 
\begin{eqnarray*} 
  \|  \phi^n - \phi^{*,n} \|_{L^\infty} 
  &=& \frac12 \| \phi^n - \phi^{n-1} \|_{L^\infty}  
 = \frac12 \| \Phi^n - \Phi^{n-1} \|_{L^\infty}   
  + \frac12 \| e^n - e^{n-1} \|_{L^\infty}    \nonumber 
\\
  &\le& 
  \frac12 A^* \dt +  C \dt^{3/2} 
\le  (\frac12 A^* + 1) \dt ,  \label{lem 2-6-a}   
\\
  \|  p^n + \dt q^n - \phi^{*,n} \|_{L^\infty} 
  &\le& \| p^n + \dt q^n - \phi^n \|_{L^\infty}  
  +  \| \phi^n - \phi^{*,n} \|_{L^\infty}   \nonumber 
\\
  &\le& 
  (A_1 + A_3  ) \dt + \frac12 A^* \dt +  C \dt^{3/2} \nonumber\\
&\le& 
 (A_1 + A_3  + \frac12 A^* + 1) \dt ,  \label{lem 2-6-b}   
\end{eqnarray*}   
in which the regularity assumption (\ref{exact-bound-1}) and the   {\it a priori}  assumption (\ref{a priori-1}) have been applied. This in turn yields 
\begin{eqnarray} 
  \| \xi^{(1)} - \phi^{*,n} \|_{L^\infty} 
  \le  (A_1 + A_3  + \frac12 A^* + 1) \dt  .  \label{lem 2-7}     
\end{eqnarray} 
To obtain a bound for $F'' (\xi^{(2)})$,  we observe that 
\begin{eqnarray} 
  \| F'' (\xi^{(2)}) \|_{L^\infty} \le \max \Big( \| F'' (p^n + \dt q^n) \|_{L^\infty} , 
  \| F'' (\phi^n) \|_{L^\infty} ,  \| F'' (\phi^{*,n}) \|_{L^\infty} \Big)  ,    \label{lem 2-8}     
\end{eqnarray} 
which comes from the range of $\xi^{(1)}$ and $\xi^{(2)}$. Meanwhile, by the fact that $F'' (\phi) = 3 \phi^2 -1$, the following bounds could be derived: 
\begin{eqnarray} 
  &&
   \| F'' (\phi^n) \|_{L^\infty}  \le 3 \| \phi^n \|_{L^\infty}^2 + 1 
   \le   C \| \phi^n \|_{H^2}^2 + 1 \le C \tilde{A}_1^2 + 1 ,   \label{lem 2-9-1}   
\\
  &&
   \| F'' (\phi^{*,n}) \|_{L^\infty}  \le 3 \| \phi^{*,n} \|_{L^\infty}^2 + 1 
   \le C \tilde{A}_1^2 + 1 ,   \label{lem 2-9-2}  
\\
  &&
   \| F'' (p^n + \dt q^n) \|_{L^\infty}  \le 3 \| p^n + \dt q^n \|_{L^\infty}^2 + 1 
   \le 3 (A_2 +1)^2 + 1 ,   \label{lem 2-9-3}       
\end{eqnarray}      
in which the   {\it a-priori}  bounds (\ref{a priori-2})-(\ref{a priori-3}) have been repeatedly applied, and the last step of (\ref{lem 2-9-3}) is based on the following estimate: 
\begin{eqnarray*} 
    \| p^n + \dt q^n \|_{L^\infty} \le  \| p^n \|_{L^\infty} + \dt \| q^n \|_{L^\infty}    
    \le A_2 + A_3 \dt \le A_2 + 1 ,  \quad \mbox{provided that $A_3 \dt \le 1$} . 
     \label{lem 2-9-4}       
\end{eqnarray*}  
Going back (\ref{lem 2-8}), we arrive at 
\begin{eqnarray} 
  \| F'' (\xi^{(2)}) \|_{L^\infty} \le A_{13} := \max \Big( C \tilde{A}_1^2 + 1 , 
  3 (A_2 +1)^2 + 1  \Big) .     \label{lem 2-10}     
\end{eqnarray} 
As a result, a substitution of (\ref{lem 2-4}), (\ref{lem 2-7}) and (\ref{lem 2-10}) into (\ref{lem 2-3}) yields 
\begin{eqnarray} 
  &&
  \| F (p^n + \dt q^n) - F (\phi^n)  - F' (\phi^{*,n}) (p^n + \dt q^n - \phi^n)  \|_{L^\infty}  \nonumber 
\\
  &\le&  
    \| F'' (\xi^{(2)}) \|_{L^\infty} \cdot \| \xi^{(1)} - \phi^{*,n} \|_{L^\infty}  
    \cdot \| p^n + \dt q^n - \phi^n \|_{L^\infty}  \nonumber 
\\
  &\le& 
  A_{14} \dt ,  \quad \mbox{with} \, \, \, 
  A_{14} =  (A_1 + A_3) (A_1 + A_3  + \frac12 A^* + 1) A_{13} .  
    \label{lem 2-11}   
\end{eqnarray}     
Finally, its combination with (\ref{lem 2-1}) implies the desired estimate:   
\begin{eqnarray*} 
     | g_n (1) | \le 
     \| F (p^n + \dt q^n) - F (\phi^n)  - F' (\phi^{*,n}) (p^n + \dt q^n - \phi^n)  \|_{L^\infty}  \cdot | \Omega |   \le A_{14} | \Omega | \dt . 
  \label{lem 2-12}   
\end{eqnarray*}  
The proof of Lemma~\ref{lem 2} is complete, by taking $A_{12} = A_{14} | \Omega|$. 
\end{proof} 

The derivative of $g_n (\eta)$ for $\eta$ around the value of $1$ is analyzed in the following lemma. 

\begin{lem}  \label{lem 3}
  Given $\phi^n$, $\phi^{n-1}$. Under the  assumption (\ref{assumption-H-0}) and the   {\it a priori}  assumption (\ref{a priori-1}), we have $ | g'_n (\eta) | \ge \frac{| S_n| }{2} \dt$ for any $1 - \dt^{1/2} \le \eta \le 1+ \dt^{1/2}$.    
\end{lem}

\begin{proof} 
Without loss of generality, we assume that $S_n <0$. A careful calculation reveals the following expression for $g'_n (\eta)$, by making use of the fact that $F (\phi) = \frac14 (\phi^2-1)^2$: 
\begin{eqnarray} 
  &&
  g_n' (\eta) = \int_\Omega \, \Big( \dt (p^n)^3 q^n + 3 \eta \dt^2 (p^n)^2  (q^n)^2 
  + 3 \eta^2 \dt^3 p^n (q^n)^3 + \eta^3 \dt^4 (q^n)^4   
  - \dt p^n q^n \nonumber 
\\
  &&  \qquad  \qquad 
  - \eta \dt^2 (q^n)^2  
  - F' (\phi^{*,n}) (p^n - \phi^n)
  - 2 \eta \dt F' (\phi^{*,n}) q^n \Big) \, d \bx 
  := \sum_{j=1}^6 I_j ,     \mbox{with}  \label{lem 3-1-a}   
\\ 
  && 
  I_1 = \dt ( (p^n)^3 - p^n, q^n ) ,  \, \, \, I_2 =  ( - F' (\phi^{*,n}), p^n - \phi^n ) , \, \, \, 
  I_3 = - 2 \eta \dt ( F' (\phi^{*,n}) , q^n )  ,   \label{lem 3-1-b}     
\\ 
  && 
  I_4 =  \eta \dt^2 ( 3 (p^n)^2 -1 , (q^n)^2  ) ,   \, \, \, 
  I_5 = 3 \eta^2 \dt^3 ( p^n , (q^n)^3  )  , \, \, \, 
  I_6 =  \eta^3 \dt^4 ( (q^n)^4 , 1 ) .   \label{lem 3-1-c}     
\end{eqnarray} 

For the $I_4$ part, an application of the $L^\infty$ bound (\ref{lem 1-0-1}) for $p^n$ and $q^n$ gives the following estimate: 
\begin{eqnarray} 
\hskip 1cm
  | I_4 | \le \eta \dt^2 ( 3 \| p^n \|_{L^\infty}^2 +1 ) \cdot \| q^n \|_{L^\infty}^2 | \Omega| 
  \le ( 3A_2^2 +1) A_3^2 | \Omega |  \eta \dt^2 
  \le ( 4A_2^2 +2) A_3^2 | \Omega |  \dt^2 ,  \label{lem 3-2}  
\end{eqnarray} 
in which the last step is based on the fact that $1 - \dt^{1/2} \le \eta \le 1+ \dt^{1/2}$. Similar bounds could be obtained for $I_5$ and $I_6$; the  details are skipped for the sake of brevity: 
\begin{eqnarray} 
  | I_5 | \le 4 A_2 A_3^3 | \Omega |  \dt^3 ,   \quad 
 | I_6 | \le 2 A_4^3 | \Omega |  \dt^4 . \label{lem 3-3}  
\end{eqnarray} 

For the $I_1$ part, we observe the following transformation: 
\begin{eqnarray} 
  I_1 &=& \dt ( (p^n)^3 - p^n, q^n ) 
     = \dt \Big( (p^n)^3 - p^n, 
     ( I + \frac34 \varepsilon^2 \dt \Delta^2 )^{-1}  \Delta F' (\phi^{*,n})  \Big)   \nonumber 
\\
  &=& 
  \dt \Big( F' (\phi^{*,n}), 
     ( I + \frac34 \varepsilon^2 \dt \Delta^2 )^{-1}  \Delta ( (p^n)^3 - p^n ) \Big)  ,     
     \label{lem 3-4-1}        
\end{eqnarray}      
in which the last step comes from the fact that $( I + \frac34 \varepsilon^2 \dt \Delta^2 )^{-1}  \Delta$ is a self-adjoint operator. Meanwhile, we introduce an approximate integral value: 
\begin{eqnarray} 
  I_1^* := 
  \dt \Big( F' (\Phi^n), 
     ( I + \frac34 \varepsilon^2 \dt \Delta^2 )^{-1}  \Delta F' (\Phi^n)  \Big)  .      
     \label{lem 3-4-2}        
\end{eqnarray}  
On the other hand, the following estimates could be derived:   
\begin{eqnarray} 
  && 
  \| F' ( \phi^{*,n}) - F' (\Phi^n) \|_{L^\infty} 
  \le \| F' ( \phi^{*,n}) - F' (\phi^n) \|_{L^\infty} 
  + \| F' ( \Phi^n) - F' (\phi^n) \|_{L^\infty} \le (A_9 + A_{10} ) \dt  ,  \nonumber
\\
  && 
  \| F' ( \phi^{*,n}) - F' (\Phi^n) \| 
  \le \| F' ( \phi^{*,n}) - F' (\Phi^n) \|_{L^\infty}  \cdot | \Omega |^{1/2} 
  \le (A_9 + A_{10} ) | \Omega |^{1/2} \dt,   \nonumber
\\
  && 
  \| F' ( \phi^{*,n}) \|_{L^\infty} 
  \le \| \phi^{*,n} \|_{L^\infty}^3 + \| \phi^{*,n} \|_{L^\infty} 
  \le A_5^3 + A_5 ,  \nonumber  
\\
  && 
  \| F' ( \phi^{*,n}) \|  \le \| F' ( \phi^{*,n}) \|_{L^\infty}  \cdot | \Omega|^{1/2} 
    \le (A_5^3 + A_5) | \Omega|^{1/2}  ,  \nonumber    
\\  
  && 
  \| \Delta \Big( F' ( \Phi^n) - F' (p^n) \Bigr) \|    
  \le \| \Delta \Big( F' ( p^n) - F' (\phi^n) \Bigr) \| 
  + \| \Delta \Big( F' ( \Phi^n) - F' (\phi^n) \Bigr) \| \le (A_7 + A_{11}  ) \dt ,  
  \nonumber
\\
  && 
  \| \Delta F' ( \Phi^n)  \| \le C ( \| \Phi^n \|_{H^2}^3 + \| \Phi^n \|_{H^2} ) 
  \le C ( (A^*)^3 + A^* ) ,  \nonumber 
\end{eqnarray}     
in which in the inequalities in (\ref{lem 1-0-2})-(\ref{lem 1-0-5}) have been extensively applied. Using the above inequalities, we get 
\begin{equation} 
\begin{aligned} 
 | I_1 - I_1^* |  
  \le & \dt \| F' ( \phi^{*,n}) - F' (\Phi^n) \|  \cdot \| F' ( \phi^{*,n}) \|   
  + \dt \| \Delta \Big( F' ( \Phi^n) - F' (p^n) \Bigr) \|    
  \cdot \| \Delta F' ( \Phi^n)  \|   
\\
  \le & 
  A_{15} \dt ^2 ,   \mbox{with} \, 
  A_{15}:= (A_5^3 + A_5)  (A_9 + A_{10} ) | \Omega |     
  + C ( (A^*)^3 + A^* )  (A_7 + A_{11}  )  .   
\end{aligned}   
  \label{lem 3-4-9}   
\end{equation}   

A similar analysis could be performed for the $I_3$ part, which could be transformed as 
\begin{eqnarray} 
  I_3 &=& -2 \eta \dt ( F' (\phi^{*,n} , q^n ) 
     = \dt \Big( (\phi^{*,n})^3 - \phi^{*,n}, 
     ( I + \frac34 \varepsilon^2 \dt \Delta^2 )^{-1}  \Delta F' (\phi^{*,n})  \Big)   \nonumber 
\\
  &=& 
  \dt \Big( F' (\phi^{*,n}), 
     ( I + \frac34 \varepsilon^2 \dt \Delta^2 )^{-1}  
     \Delta ( (\phi^{*,n})^3 - \phi^{*,n} ) \Big)  .     
     \label{lem 3-5-1}        
\end{eqnarray}     
On the other hand, we introduce an approximate integral value: 
\begin{eqnarray} 
  I_3^* := -2 \eta I_1^* 
  = -2 \eta \dt \Big( F' (\Phi^n), 
     ( I + \frac34 \varepsilon^2 \dt \Delta^2 )^{-1}  \Delta F' (\Phi^n)  \Big)  .      
     \label{lem 3-5-2}        
\end{eqnarray}  
The following estimate could be derived in the same manner; the  details are left to interested readers: 
\begin{eqnarray} 
 | I_3 - I_3^* |  
  \le A_{16} \dt ^2,  \, \, 
  A_{16}:= (A_5^3 + A_5)  (A_9 + A_{10} ) | \Omega |     
  + C ( (A^*)^3 + A^* )  (A_9 + A_{11}  )  .  \label{lem 3-5-3}   
\end{eqnarray}     

  For the $I_2$ part, we begin with the following expression, which comes from (\ref{lem 1-1-1-a}): 
\begin{eqnarray} 
  I_2 =  ( - F' (\phi^{*,n}), p^n - \phi^n ) 
  =   \varepsilon^2 \dt \Big( F' (\phi^{*,n}),  
  ( I + \frac34 \varepsilon^2 \dt \Delta^2 )^{-1}  
  \Delta^2 ( \frac34 \phi^n + \frac14  \phi^{n-1} ) \Big) .   
   \label{lem 3-6-1}   
\end{eqnarray}  
Similarly, we also introduce an approximate integral value: 
\begin{eqnarray} 
  I_2^* :=   \varepsilon^2 \dt \Big( F' (\phi^{*,n}),  
  ( I + \frac34 \varepsilon^2 \dt \Delta^2 )^{-1}  
  \Delta^2 \Phi^n \Big) .   
   \label{lem 3-6-2}   
\end{eqnarray} 
To estimate the difference between $I_1$ and $I_2^*$, we have the following observations: 
\begin{eqnarray} 
  && 
   \| \Delta^2 ( \frac34 \phi^n + \frac14  \phi^{n-1} )  
   -    \Delta^2 ( \frac34 \Phi^n + \frac14  \Phi^{n-1} ) \|    
   = \| \Delta^2 ( \frac34 e^n + \frac14  e^{n-1} )  \|  
   \le \dt^{3/2} , \, \, \mbox{(by (\ref{a priori-1}))} ,   \nonumber
 \\
   && 
   \|  \Delta^2 ( \frac34 \Phi^n + \frac14  \Phi^{n-1} ) - \Delta^2 \Phi^n \|    
   = \frac14 \| \Delta^2 ( \Phi^n - \Phi^{n-1} )  \|  
   \le A^* \dt , \, \, \mbox{(by (\ref{exact-bound-1}))} ,   \nonumber  
\end{eqnarray}  
so that 
\begin{eqnarray} 
  && 
  \| \Delta^2 ( \frac34 \phi^n + \frac14  \phi^{n-1} ) - \Delta^2 \Phi^n \|  \le
    \| \Delta^2 ( \frac34 \phi^n + \frac14  \phi^{n-1} )  \nonumber\\
   &&-    \Delta^2 ( \frac34 \Phi^n + \frac14  \Phi^{n-1} ) \|    
   +  \|  \Delta^2 ( \frac34 \Phi^n + \frac14  \Phi^{n-1} ) - \Delta^2 \Phi^n \|    
   \le ( A^* + 1) \dt   .   \label{lem 3-6-5}        
\end{eqnarray}  
Moreover, because of the fact that, all the eigenvalues associated with the operator $( I + \frac34 \varepsilon^2 \dt \Delta^2 )^{-1}$ is bounded by 1, we see that 
\begin{eqnarray}  
  \| ( I + \frac34 \varepsilon^2 \dt \Delta^2 )^{-1} 
  \Delta^2 ( \frac34 \phi^n + \frac14  \phi^{n-1} ) 
  - ( I + \frac34 \varepsilon^2 \dt \Delta^2 )^{-1} \Delta^2 \Phi^n \|    
   \le ( A^* + 1) \dt   .   \label{lem 3-6-6}        
\end{eqnarray}  
In addition, the fact that $\| \Delta^2 \Phi^n \| \le A^*$ gives 
\begin{eqnarray*}  
  \| ( I + \frac34 \varepsilon^2 \dt \Delta^2 )^{-1} \Delta^2 \Phi^n \|    
   \le \| \Delta^2 \Phi^n \|  \le A^*   .   \label{lem 3-6-7}        
\end{eqnarray*}  
Then we arrive at the following estimate: 
\begin{eqnarray} 
 | I_2 - I_2^* |  
  &\le& \dt \| F' ( \phi^{*,n}) - F' (\Phi^n) \|  \cdot \| F' ( \phi^{*,n}) \|    \nonumber 
\\
  && 
  + \dt \| ( I + \frac34 \varepsilon^2 \dt \Delta^2 )^{-1} 
  \Delta^2 ( \frac34 \phi^n + \frac14  \phi^{n-1} -  \Phi^n ) \|  
  \cdot \| ( I + \frac34 \varepsilon^2 \dt \Delta^2 )^{-1} \Delta^2 \Phi^n \|   \nonumber 
\\
  &\le& 
  A_{17} \dt ^2 ,  \quad \mbox{with} \, \, 
  A_{17}:= (A_5^3 + A_5)  (A_9 + A_{10} ) | \Omega |     
  + A^* (A^* +1 )  .  \label{lem 3-6-8}   
\end{eqnarray}    

As a consequence, a substitution of (\ref{lem 3-2}), (\ref{lem 3-3}), (\ref{lem 3-4-9}), (\ref{lem 3-5-3}) and (\ref{lem 3-6-8}) into (\ref{lem 3-1-a})-(\ref{lem 3-1-b}) indicates that 
\begin{eqnarray} 
\hskip 1cm
  g_n' (\eta) = I_1^* + I_2^* + I_3^* + {\cal R}_1 ,  \quad 
  \, \, \, 
  | {\cal R}_1 | \le  \Big( ( 4A_2^2 +2) A_3^2 | \Omega |  
  + A_{15} + A_{16} + A_{17} + 1 \Big) \dt^2  ,  \label{lem 3-7-1}   
\end{eqnarray} 
provided that $4 A_2 A_3^3 | \Omega |  \dt \le \frac12$, $2 A_4^3 | \Omega |  \dt^2 \le \frac12$. In addition, a careful calculation reveals that 
\begin{eqnarray} 
   I_1^* + I_2^* + I_3^* 
   =  \dt \Big( F' (\Phi^n), 
     ( I + \frac34 \varepsilon^2 \dt \Delta^2 )^{-1}  
     \Big( ( 1- 2 \eta) \Delta F' (\Phi^n) + \varepsilon^2 \Phi^n \Big) \Big) . 
     \label{lem 3-7-2}          
\end{eqnarray} 
A further decomposition is made to facilitate the later analysis: 
\begin{eqnarray} 
   I_1^* + I_2^* + I_3^*  = I_7^* + I_8^* ,   \, \, \,  \mbox{with} &&  
  I_7^* = \dt \Big( F' (\Phi^n), 
     ( I + \frac34 \varepsilon^2 \dt \Delta^2 )^{-1}  
     \Big( - \Delta F' (\Phi^n) + \varepsilon^2 \Delta^2 \Phi^n \Big) \Big)  ,  \nonumber 
 \\
   &&
   I_8^* = 2 (1- \eta) \dt \Big( F' (\Phi^n), 
     ( I + \frac34 \varepsilon^2 \dt \Delta^2 )^{-1}  \Delta F' (\Phi^n) \Big) . 
     \label{lem 3-7-3}          
\end{eqnarray}   
An estimate for $I_8^*$ is straightforward: 
\begin{eqnarray*} 
  && 
   \| F' (\Phi^n) \|  , \, \,  \| \Delta F' (\Phi^n)  \| 
   \le C ( \| \Phi^n \|_{H^2}^3 + \| \Phi^n ||_{H^2} ) 
   \le C ((  A^*)^3 + A^* )  ,    \label{lem 3-7-4}    
\\
  && 
   \| ( I + \frac34 \varepsilon^2 \dt \Delta^2 )^{-1} \Delta F' (\Phi^n)  \| 
   \le \|  \Delta F' (\Phi^n)  \|    
   \le C ((  A^*)^3 + A^* )  ,    \label{lem 3-7-5}                   
\end{eqnarray*}  
so that 
\begin{eqnarray} 
\hskip 1cm
   | I_8^* | \le 2 | 1- \eta | \dt\cdot  \| F' (\Phi^n) \| 
     \cdot \| ( I + \frac34 \varepsilon^2 \dt \Delta^2 )^{-1}  \Delta F' (\Phi^n) \| 
     \le  C ((  A^*)^3 + A^* )^2 \dt^{3/2}  ,    
     \label{lem 3-7-6}          
\end{eqnarray}   
in which the fact that $| 1 - \eta | \le \dt^{1/2}$ has been applied in the last step. For the part $I_7^*$, we observe that $- \Delta F' (\Phi^n) + \varepsilon^2 \Delta^2 \Phi^n = - ( \Phi_t)^n$ (satisfied by the original PDE), so that a further decomposition is available:  
\begin{eqnarray} 
  &&
   I_7^* = - \dt \Big( F' (\Phi^n), 
     ( I + \frac34 \varepsilon^2 \dt \Delta^2 )^{-1}  ( \Phi_t)^n \Big) 
     := I_9^* + I_{10}^* ,  \quad \mbox{with}  \nonumber 
 \\
   && 
   I_9^* = - \dt \Big( F' (\Phi^n), (\Phi_t)^n \Big) ,  \quad 
   I_{10}^* = \frac34 \varepsilon^2 \dt^2 \Big( F' (\Phi^n), 
     ( I + \frac34 \varepsilon^2 \dt \Delta^2 )^{-1}  \Delta^2 ( \Phi_t)^n \Big)  . 
     \label{lem 3-7-7}         
\end{eqnarray} 
By (\ref{assumption-H-0}), an exact value is available for $I_9^*$: 
\begin{eqnarray} 
  I_9^* = - S_n \dt .  
  \label{lem 3-7-8}      
\end{eqnarray} 
The estimate for $I_{10}^*$ could be carried out as follows
\begin{eqnarray*} 
  &&
  \| F' (\Phi^n) \| \le C ((  A^*)^3 + A^* )  ,  \quad 
  \| \Delta^2 ( \Phi_t)^n \| \le C A^* ,  \quad \mbox{by (\ref{exact-bound-1}) } ,  
  \label{lem 3-7-9}   
\\
   &&
   \| ( I + \frac34 \varepsilon^2 \dt \Delta^2 )^{-1}  \Delta^2 ( \Phi_t)^n  \| 
   \le \| F' (\Phi^n) \|  \le C A^* ,   \label{lem 3-7-10}   
\end{eqnarray*} 
so that 
\begin{eqnarray} 
   | I_{10}^* | \le \frac34 \varepsilon^2 \dt^2  \| F' (\Phi^n) \| 
     \cdot \| ( I + \frac34 \varepsilon^2 \dt \Delta^2 )^{-1}  \Delta^2 ( \Phi_t)^n \|  
     \le C ((  A^*)^4 + A^2 ) \varepsilon^2 \dt^2 .    
     \label{lem 3-7-11}         
\end{eqnarray} 
Therefore, a substitution of (\ref{lem 3-7-6}), (\ref{lem 3-7-7}), (\ref{lem 3-7-8}), (\ref{lem 3-7-11}) into (\ref{lem 3-7-3}) results in 
\begin{eqnarray} 
  && 
   I_1^* + I_2^* + I_3^*  = - S_n \dt + {\cal R}_2 ,   \nonumber 
\\
  && 
   \mbox{with}  \, \, \,  
   | {\cal R}_2 | \le C ((  A^*)^4 + A^2 ) \varepsilon^2 \dt^2    
       + C ((  A^*)^3 + A^* )^2 \dt^{3/2}  \le \dt^{5/4} ,  \label{lem 3-7-12}         
\end{eqnarray}        
provided that $\dt$ is sufficiently small. Its combination with (\eqref{lem 3-7-1}) leads to 
\begin{eqnarray*} 
   g_n' (\eta) = - S_n \dt + {\cal R} ,  \quad 
  \mbox{with} \, \, \,  {\cal R} = {\cal R}_1 + {\cal R}_2 , \,  \, \, 
   | {\cal R} | \le 2 \dt^{5/4} ,  \label{lem 3-7-13}         
\end{eqnarray*}        
for any $1 - \dt^{1/2} \le \eta \le 1 + \dt^{1/2}$. Subsequently, if we take $\dt$ small enough with $2 \dt^{1/4} \le \frac{S_n}{2}$, the following estimate is valid: 
\begin{equation*} 
  \frac{|S_n|}{2} \dt \le g_n' (\eta) \le \frac{3 |S_n|}{2} \dt ,  \, \, \, 
  \mbox{so that} \, \, \, | g_n' (\eta)  | \ge \frac{|S_n|}{2} \dt ,  \, \, \, 
  \mbox{for $1 - \dt^{1/2} \le \eta \le 1 + \dt^{1/2}$} . 
    \label{lem 3-7-14}         
\end{equation*}        
 This finishes the proof of Lemma~\ref{lem 3}.       
\end{proof}

\subsection{Proof of Theorem~\ref{thm: solvability}}

We can now proceed to the proof of Theorem~\ref{thm: solvability}. 
  Without loss of generality, we assume that $S_n <0$ and $g_n(1) > 0$, as the other cases could be analyzed in the same way. By Lemmas~\ref{lem 2} and \ref{lem 3}, we have
\begin{eqnarray} 
   0 \le g_n (1) \le A_{12} \dt^2 ,  \quad 
   g_n' (\eta) \ge - \frac{S_n}{2} \dt ,  \, \, \, 
   \mbox{for $1 - \dt^{1/2} \le \eta \le 1 + \dt^{1/2}$}  . 
   \label{solvability-1} 
\end{eqnarray} 
Then we conclude that, $g_n (\eta)$ is monotone and increasing over the interval $(1 - \delta^*, 1)$, with $\delta^* = - \frac{2 A_{12}}{S_n} \dt$, so that 
\begin{eqnarray} 
  g_n(1) \ge 0 , \, \, \, g_n (1 - \delta^*) \le g_n(1) + \frac{S_n}{2} \dt \cdot \delta^* 
   \le A_{12} \dt^2  -   A_{12} \dt^2 = 0 .  \label{solvability-2} 
\end{eqnarray}  
Therefore, there is a unique solution for $g_n (\eta) =0$ over the interval $(1 - \delta^*, 1)$. In addition, we have $\delta^* \le \dt^{1/2}$, if $\dt \leq \frac{S^4_n}{256} \leq \frac{S_n^2}{4A_{12}} $ with  $|S_n| \ll 1$. Since $g_n (\eta)$ is increasing over $[1- \dt^{1/2}, 1 + \dt^{1/2}]$, such a solution is unique in the interval $[1- \dt^{1/2}, 1 + \dt^{1/2}]$.     The proof is completed.

\section{Error analysis}\label{sec:convergence}
Given a tolerance $\gamma>0$, we shall assume
\begin{equation} 
	\Big| \int_\Omega \, F' (\Phi) \Phi_t \, d \bx \Big| \ge \gamma, \text{ for } 0\le t\le T,
	\label{assumptionb} 
\end{equation} 
which implies in particular that $|S_n|\ge \gamma$ for all $n\le T/\Delta t$.

Usually,  the above assumption is only satisfied for  $0\le t\le T_0\le T$,  but for the sake of simplicity we shall assume $T_0=T$. An error analysis is carried out in this section for the scheme  (\ref{scheme-SAVN-1})-(\ref{scheme-SAVN-2}) under the above assumption. For the time intervals where \eqref{assumptionb}  is not satisfied, we use the modified Lagrange multiplier approach by setting $\eta^{n+1/2}=1$ and computing $\phi^{n+1}$ by (\ref{scheme-SAVN-1}).


The main result of this section is stated in the following theorem. 

\begin{thm}
	\label{thm: convergence}
	Given initial data $\Phi^0 \in H^8_{\rm per}(\Omega)$, suppose the exact solution for the Cahn-Hilliard equation~(\ref{equation-CH}) is of regularity class $\mathcal{R}$, 
	and satisfies the  assumption (\ref{assumption-H-0}). 
	Then, provided that $\dt \le C\min \Big( (  \hat{C} )^{-2} ,  \hat{C}^2 \varepsilon^2, \frac{\gamma^4}4 \Big)$, we have
	\begin{equation}
		\| \Delta^2 e^m \| +  \Big( \varepsilon^2 \dt   \sum_{k=1}^{m} \| \Delta^3 e^k \|^2 \Big)^{1/2}  \le \hat{C} \dt^2,   \label{convergence-0}
	\end{equation}
	for all positive integers $m$, such that $t^m=m\dt \le T$, where $\hat{C}, C>0$  are some positive constants independent of $n$ and $\dt$.
\end{thm}
The proof of this theorem will be carried out through a series of intermediate estimates which we describe below.

\subsection{Energy stability and the $H^1$ estimate} 

The following result could be proved in a straightforward way. 

\begin{lem}  \label{prop: energy stab} 
If the numerical system (\ref{scheme-SAVN-1})-(\ref{scheme-SAVN-2}) is solvable, it is energy stable with respect to the modified energy functional: $\tilde{E} (\phi^{n+1} , \phi^n) := E (\phi^{n+1}) + \frac18 \varepsilon^2 \| \nabla ( \phi^{n+1} - \phi^n ) \|^2$, i.e., the following energy inequality holds: 
\begin{equation} 
    \tilde{E} (\phi^{n+1} , \phi^n) \le \tilde{E} (\phi^n , \phi^{n-1}) ,  \quad \forall n \ge 0  .  
    \label{energy stab-0} 
\end{equation}     
\end{lem}

\begin{proof} 
Taking an inner product with (\ref{scheme-SAVN-1}) by $(-\Delta)^{-1} (\phi^{n+1} - \phi^n)$ gives  
\begin{equation} 
   ( F' (\phi^{*,n}) \eta^{n+1/2} , \phi^{n+1} - \phi^n )    
  - \varepsilon^2  \Big( \Delta ( \frac34 \phi^{n+1} + \frac14 \phi^{n-1} )  , 
    \phi^{n+1} - \phi^n \Big)   = - \frac{1}{\dt} \| \phi^{n+1} - \phi^n \|_{H^{-1}}^2 . 
    \label{energy stab-1} 
\end{equation}         
Meanwhile, the following estimates are available: 
\begin{eqnarray*} 
  &&
   ( F' (\phi^{*,n}) \eta^{n+1/2} , \phi^{n+1} - \phi^n )   
    = \Big( F (\phi^{n+1} ) - F (\phi^n) , 1 \Big)   ,  \quad 
    \mbox{(by \eqref{scheme-SAVN-2})} ,   \label{energy stab-2-1}    
\\
  && 
     - \Big( \Delta ( \frac34 \phi^{n+1} + \frac14 \phi^{n-1} )  , 
    \phi^{n+1} - \phi^n \Big)  
    =  \Big( \nabla ( \frac34 \phi^{n+1} + \frac14 \phi^{n-1} )  , 
    \nabla ( \phi^{n+1} - \phi^n)  \Big)     \nonumber 
\\
  && \qquad 
   \ge \frac12 ( \| \nabla \phi^{n+1} \|^2 - \| \nabla \phi^n \|^2 ) 
   + \frac18 ( \| \nabla ( \phi^{n+1} - \phi^n ) \|^2 
   - \| \nabla ( \phi^n - \phi^{n-1} ) \|^2  )   . 
   \label{energy stab-2-2} 
\end{eqnarray*}  
Going back (\ref{energy stab-1}), we arrive at the desired inequality (\ref{energy stab-0}).               
\end{proof} 

As a result of this result, we obtain a uniform bound of the original energy functional $E (\phi^n)$, for any $n \ge 1$: 
\begin{eqnarray*} 
   E (\phi^n ) \le \tilde{E} (\phi^n , \phi^{n-1})  \le ... 
   \le \tilde{E} (\phi^0, \phi^{-1}) = E (\phi^0) := C_0 ,  \quad 
   \mbox{by taking $\phi^{-1} = \phi^0$} .  \label{H1 est-1} 
\end{eqnarray*}     
Meanwhile, since $E (\phi^n) \ge \frac{\varepsilon^2}{2} \| \nabla \phi^n \|^2$, we get 
\begin{eqnarray*} 
   \frac{\varepsilon^2}{2} \| \nabla \phi^n \|^2 \le E (\phi^n ) \le C_0 ,  \quad 
   \mbox{so that}  \, \, 
   \| \nabla \phi^n \| \le (2 C_0)^{1/2} \varepsilon^{-1} .   
     \label{H1 est-2} 
\end{eqnarray*}
And also, the numerical solution (\ref{scheme-SAVN-1})-(\ref{scheme-SAVN-2}) is mass conservative, so that 
\begin{eqnarray*} 
  \overline{\phi^n} = \overline{\phi^{n-1}} = ... = \overline{\phi^0} := \beta_0 ,  \quad 
  \forall n \ge 0 . \label{H1 est-3} 
\end{eqnarray*}
In turn, an application of Poincar\'e inequality yields a uniform in time $H^1$ bound for the numerical solution:  
\begin{eqnarray} 
  \| \phi^n \|_{H^1} \le C \Big( | \overline{\phi^n} |  + \| \nabla \phi^n \|  \Big) 
  \le C_1 := C ( | \beta_0 | + (2 C_0)^{1/2} \varepsilon^{-1} ) , \quad 
  \forall n \ge 0 . \label{H1 est-4} 
\end{eqnarray}  
We notice that $C_1$ is uniform in time, while it depends on $\varepsilon^{-1}$ in a polynomial pattern.

\subsection{The $\ell^\infty(0,T; H^2)$ estimate for the numerical solution} 

\begin{lem}  \label{prop: H2 est} 
Under the  assumption (\ref{assumption-H-0}) with $\gamma \ge 4\dt^{\frac 14}$ and the   {\it a priori}  assumption (\ref{a priori-1}), we have 
\begin{equation} 
    \| \phi^n \|_{H^2} \le C_2 ,  
    \quad \forall n \ge 1 ,  
    \label{H2 est-0} 
\end{equation}     
in which $C_2$ depends on $\varepsilon^{-1}$ in a polynomial pattern, while it is independent on $\dt$ and $T$. 
\end{lem}

\begin{proof} 
	First, it is noticed that, by the assumptions and Theorem 3.1, the scheme (\ref{scheme-SAVN-1})-(\ref{scheme-SAVN-2}) has a unique solution in $1 - \dt^{1/2} \le \eta \le 1+ \dt^{1/2}$.
 Taking an inner product with (\ref{scheme-SAVN-1}) by $2 \Delta^2 (\phi^{n+1} - \phi^n)$, we obtain 
\begin{eqnarray} 
  &&
  \| \Delta \phi^{n+1} \|^2 - \| \Delta \phi^n \|^2 
  + \| \Delta ( \phi^{n+1} - \phi^n ) \|^2  
  +  \varepsilon^2 \dt ( \Delta^2 \phi^{n+1}, 
  \Delta^2 ( \frac32 \phi^{n+1} + \frac12 \phi^{n-1} ) )   \nonumber 
\\
  &=& 
   2 \eta^{n+1/2} \dt ( \Delta F' (\phi^{*,n}) , \Delta^2 \phi^{n+1}  )   . 
    \label{H2 est-1} 
\end{eqnarray} 
The inner product associated surface diffusion could be analyzed as follows: 
	\begin{eqnarray}
 && (  \Delta^2 \phi^{n+1}  ,  \Delta^2 ( \frac32 \phi^{n+1} + \frac12 \phi^{n-1} ) ) 
   =  \frac32  \| \Delta^2 \phi^{n+1} \|^2   
   + \frac12  (  \Delta^2 \phi^{n+1}  , \Delta^2 \phi^{n-1} )
 	\nonumber  
	\\ 
&&\ge  \frac32  \|  \Delta^2 \phi^{n+1} \|^2  
- \frac12  \Big( \frac12  \| \Delta^2 \phi^{n+1} \|^2 
 +\frac12 \| \Delta^2 \phi^{n-1} \|^2  \Big)
	\nonumber 
\ge  \frac54  \|  \Delta^2 \phi^{n+1} \|^2 -  \frac14 \| \Delta^2 \phi^{n-1} \|^2 .
	\label{H2 est-2}    
	\end{eqnarray}
For the right hand side in \eqref{H2 est-1}, we recall the expansion (\ref{lem 1-1-6}) and apply the H\"older inequality: 
\begin{eqnarray} 
   \| \Delta F' (\phi^{*,n}) \| \le ( 3 \| \phi^{*,n} \|_{L^\infty}^2 +1 ) \| \Delta \phi^{*,n} \|  
   + 6 \| \phi^{*,n} \|_{L^\infty} \cdot \| \nabla \phi^{*,n} \|_{L^4}^2 . 
   \label{H2 est-3-1}    
\end{eqnarray}   
Meanwhile, the following estimates are available, based on Sobolev embedding and weighted inequalities, as well as the uniform in time estimate (\ref{H1 est-4}):  
\begin{eqnarray} 
  \| \phi^{*,n} \|_{L^\infty} 
  &\le& C \|  \phi^{*,n} \|_{H^1}^{5/6} \cdot  \|  \phi^{*,n} \|_{H^4}^{1/6}  
  \le C \|  \phi^{*,n} \|_{H^1}^{5/6} 
   \cdot  ( \|  \phi^{*,n} \|_{H^1} + \|  \Delta^2 \phi^{*,n} \| )^{1/6}   \nonumber 
\\
  &\le&
    C C_1^{5/6} \cdot  ( C_1 + \|  \Delta^2 \phi^{*,n} \| )^{1/6}  ,  \label{H2 est-3-2}     
\\
  \| \Delta \phi^{*,n} \|
  &\le&  \| \nabla \phi^{*,n} \|^{2/3} \cdot  \|  \Delta^2 \phi^{*,n} \|^{1/3}  
  \le C C_1^{2/3} \cdot \|  \Delta^2 \phi^{*,n} \|^{1/3}  ,  \label{H2 est-3-3}    
\\
  \| \nabla \phi^{*,n} \|_{L^4} 
  &\le&  C \|  \nabla \phi^{*,n} \|_{H^{3/4}} 
  \le C \|  \nabla \phi^{*,n} \|^{3/4} \cdot  \|  \Delta \phi^{*,n} \|_{H^4}^{1/4}  
  \le C C_1^{3/4} \cdot \|  \Delta^2 \phi^{*,n} \|^{1/4}  .  \label{H2 est-3-4}       
\end{eqnarray} 
Going back (\ref{H2 est-3-1}), we get   
\begin{eqnarray} 
   \| \Delta F' (\phi^{*,n}) \| \le 
   C C_1^3 + C C_1^{7/3} \cdot \|  \Delta^2 \phi^{*,n} \|^{2/3}    
   + C C_1^{2/3} \cdot \|  \Delta^2 \phi^{*,n} \|^{1/3}   . 
   \label{H2 est-3-5}    
\end{eqnarray}   
In turn, an application of Cauchy inequality implies that 
\begin{eqnarray} 
  &&
   2 \eta^{n+1/2} ( \Delta F' (\phi^{*,n}) , \Delta^2 \phi^{n+1}  )  
   \le 3  \| \Delta F' (\phi^{*,n}) \| \cdot \|  \Delta^2 \phi^{n+1}  \|   \nonumber 
\\
  &\le& 
    9 \varepsilon^{-2} \| \Delta F' (\phi^{*,n}) \|^2    
   + \frac14 \varepsilon^2  \|  \Delta^2 \phi^{n+1}  \|^2   \nonumber 
\\
  &\le& 
  C \varepsilon^{-2} ( C_1^6 + C_1^{14/3} \cdot \|  \Delta^2 \phi^{*,n} \|^{4/3}    
   + C_1^{4/3} \cdot \|  \Delta^2 \phi^{*,n} \|^{2/3}  ) 
   + \frac14 \varepsilon^2  \|  \Delta^2 \phi^{n+1}  \|^2   \nonumber 
\\
  &\le& 
  C_{\varepsilon}^{(1)}  
  +  \frac{\varepsilon^2}{10} \|  \Delta^2 \phi^{*,n} \|^2     
   + \frac14 \varepsilon^2  \|  \Delta^2 \phi^{n+1}  \|^2    \nonumber 
\\
  &\le& 
  C_{\varepsilon}^{(1)}  
  +  \frac{9 \varepsilon^2}{20} \|  \Delta^2 \phi^{n} \|^2     
  +  \frac{\varepsilon^2}{20} \|  \Delta^2 \phi^{n-1} \|^2       
   + \frac14 \varepsilon^2  \|  \Delta^2 \phi^{n+1}  \|^2  ,  \label{H2 est-3-6}      
\end{eqnarray}     
in which the Young's inequality has been extensively applied in the fourth step, and the last step is based on the fact that $\|  \Delta^2 \phi^{*,n} \|^2  \le  
\frac92 \|  \Delta^2 \phi^{n} \|^2  +  \frac12 \|  \Delta^2 \phi^{n-1} \|^2$. We also notice that $C_{\varepsilon}^{(1)} $ depends on $C_1$ and $\varepsilon^{-1}$ in a polynomial pattern.    

Subsequently, a substitution of (\ref{H2 est-2}) and (\ref{H2 est-3-6}) into (\ref{H2 est-1}) leads to 
\begin{equation*} 
  \| \Delta \phi^{n+1} \|^2 - \| \Delta \phi^n \|^2  
  +  \varepsilon^2 \dt \| \Delta^2 \phi^{n+1} \|^2  
  \le  
  C_{\varepsilon}^{(1)}  \dt 
  +  \frac{9 \varepsilon^2}{20} \dt \|  \Delta^2 \phi^{n} \|^2     
  +  \frac{3 \varepsilon^2}{10} \dt \|  \Delta^2 \phi^{n-1} \|^2   . 
    \label{H2 est-4-1} 
\end{equation*} 
In fact, this inequality could be rewritten as 
\begin{equation*} 
\begin{aligned} 
  &
  \| \Delta \phi^{n+1} \|^2 
  +  \frac45 \varepsilon^2 \dt \| \Delta^2 \phi^{n+1} \|^2  
  +  \frac{3}{10} \varepsilon^2 \dt \| \Delta^2 \phi^n \|^2    
  +  \frac15 \varepsilon^2 \dt \| \Delta^2 \phi^{n+1} \|^2  
  +  \frac{1}{20} \varepsilon^2 \dt \| \Delta^2 \phi^n \|^2     
\\
  \le&  
  \| \Delta \phi^n \|^2    
  +  \frac45 \varepsilon^2 \dt \| \Delta^2 \phi^n \|^2  
  +  \frac{3}{10} \varepsilon^2 \dt \| \Delta^2 \phi^{n-1} \|^2   
   + C_{\varepsilon}^{(1)}  \dt  . 
\end{aligned} 
  \label{H2 est-4-2} 
\end{equation*} 
With an introduction of a modified $H^2$ energy
\begin{eqnarray*} 
  G^n := \| \Delta \phi^n \|^2    
  +  \frac45 \varepsilon^2 \dt \| \Delta^2 \phi^n \|^2  
  +  \frac{3}{10} \varepsilon^2 \dt \| \Delta^2 \phi^{n-1} \|^2 , 
  \label{H2 est-4-3} 
\end{eqnarray*}     
we get 
\begin{eqnarray*} 
  G^{n+1} + B_0 \varepsilon^2 \dt G^{n+1} \le  G^n 
   + C_{\varepsilon}^{(1)}  \dt  ,  
    \label{H2 est-4-4} 
\end{eqnarray*} 
in which the following elliptic regularity has been used: 
\begin{eqnarray*} 
  B_1 \| \Delta f \|^2 \le \| \Delta^2 f \|^2 ,  \quad \mbox{and}  \, \, \, 
  B_0 = \min \Big( \frac{B_1}{10} , \frac18 \Big) .  \label{H2 est-4-5} 
\end{eqnarray*}   
An application of the induction argument implies that 
\begin{eqnarray*} 
   G^{n+1} \le  ( 1 + B_0 \varepsilon^2 \dt )^{-(n+1)} G^0 
   +  \frac{C_{\varepsilon}^{(1)} }{B_0 \varepsilon^2} .  \label{H2 est-4-6} 
\end{eqnarray*}   
Of course, we could introduce a uniform in time quantity $B_2^* := G^0 
   +  \frac{C_{\varepsilon}^{(1)} }{B_0 \varepsilon^2}$, so that $\| \Delta \phi^n \|^2 \le G^n \le B_2^*$ for any $n \ge 0$. In turn, an application of the elliptic regularity reveals that 
\begin{eqnarray*} 
     \| \phi^n \|_{H^2} \le C \Big( | \overline{\phi^k} | + \| \Delta \phi^k \| \Big) 
     \le C ( | \beta_0 | + ( B_2^* )^{1/2} ) := C_2 ,  \quad \forall n \ge 0 . \label{H2 est-4-7} 
\end{eqnarray*}   
in which the uniform in time constant $C_2$ depends on $\varepsilon^{-1}$ in a polynomial pattern. This finishes the proof of Proposition~\ref{prop: H2 est}. 
\end{proof} 

Using similar tools, a uniform-in-time $H^k$ bound for the numerical solution could be established, for any $k \ge 3$, by taking inner product with (\ref{scheme-SAVN-1}) by $2 (-\Delta)^k \phi^{n+1}$, and performing the associated estimates. The details are left to interested readers.

\begin{lem}  \label{prop: Hm est} 
Under the assumptions of Lemma ~\ref{prop: H2 est}, we have for any $k\ge 3$,
\begin{equation} 
    \| \phi^n \|_{H^k} \le C_k ,  
    \quad \forall n \ge 1  ,  
    \label{Hm est-0} 
\end{equation}     
in which $C_k$ depends on $\varepsilon^{-1}$ in a polynomial power, while it is independent on $\dt$ and $T$. 
\end{lem}

\subsection{Estimate for $\| \phi^{m+1} - \phi^m \|_{H^\ell}$ } 

The following estimate is needed in the later analysis. 

	\begin{lem} \label{prop: stab-time-1}  
Under the assumptions of Lemma ~\ref{prop: H2 est}, we have 
	\begin{eqnarray}
\max_{0\le m \le M-1} \| \phi^{m+1} - \phi^m \|_{H^\ell} \le D_{\ell} \dt  ,
	\label{time-stability-1-bound}
	\end{eqnarray}
where $D_\ell>0$ is a constant independent of $\dt$ and $T$, dependent on $\varepsilon^{-1}$ in a polynomial style. 
	\end{lem}

\begin{proof}
The numerical scheme \eqref{scheme-SAVN-1} implies that 
 	\begin{equation}
 \| \phi^{m+1}-\phi^m \|_{H^\ell} = \dt \| \Delta \mu^{m+1/2}  \|_{H^\ell} , 
  \quad  \mu^{m+1/2} = F' (\phi^{*,m}) \eta^{m+1/2} 
  - \varepsilon^2  \Delta ( \frac34 \phi^{m+1} + \frac14 \phi^{m-1} )  . 
   \label{est-time-1-1} 
    \end{equation}    
By the expansion for $\mu^{m+1/2}$, the following estimates could be derived, helped by repeated applications of the H\"older inequality and Sobolev embedding: 
	\begin{align} 
\nrm{ \Delta F' (\phi^{*,m}) }_{H^\ell} \le& C \nrm{  F' (\phi^{*,m})  }_{H^{\ell+2} } 
   \le C  \nrm{ \phi^{*,m}  }_{H^{\ell+2} }^3    \nonumber 
\\
  \le& C  \left( \nrm{ \phi^m }_{H^{\ell+2} }^3 
+  \nrm{ \phi^{m-1} }_{H^{\ell+2} }^3  \right) \le C C_{\ell+2}^3 , 
	\label{est-time-1-3} 
	\\
\nrm{ \Delta^2 ( \frac34 \phi^{m+1} + \frac14 \phi^{m-1} )  }_{H^\ell} \le& C \left(  \nrm{ \phi^{m+1} }_{H^{\ell+4} } +  \nrm{ \phi^{m-1} }_{H^{\ell+4} }  \right) 
\le C C_{\ell+4} , 
	\label{est-time-1-5} 
	\end{align}
with the uniform-in-time estimates (\ref{Hm est-0}) used. Consequently, a substitution of (\ref{est-time-1-3})-(\ref{est-time-1-5}) into (\ref{est-time-1-1}) yields 
	\begin{equation*} 
\nrm{ \phi^{m+1}-\phi^m }_{H^\ell} \le  D_{\ell} \dt ,  \quad    
   D_{\ell} := C  \left(  C_{\ell+2}^3 + \varepsilon ^2 C_{\ell+4}  \right) .  
	\label{est-time-1-6} 
    \end{equation*} 
Note that $D_{\ell}$ depends on $\varepsilon^{-1}$ in a polynomial form, since both $C_{\ell+2}$ and $C_{\ell+4}$ do as well. This completes the proof of Proposition~\ref{prop: Hm est}. 
	\end{proof}

\subsection{A refined estimate for $\eta^{n+1/2}$} 

An $O (\dt^{1/2})$ estimate have been derived for $| \eta^{n+1/2} |$ in Theorem~\ref{thm: solvability}. 
 In fact, by the estimate (\ref{solvability-1}), we see that
 \begin{eqnarray} 
 	| \eta^{n+1/2} -1 |  \le \frac{A_{12} \dt^2}{\frac{ |S_n| }{2} \dt} 
 	=  \frac{2 A_{12} \dt}{|S_n|}\le  \frac{2 A_{12} \dt}{\gamma} . 
 	\label{eta est-prelim-2} 
 \end{eqnarray} 
Again, this rough estimate is not sufficient for an optimal rate of convergence. We aim to improve this estimate so that $\eta^{n+1/2} -1 = O (\dt^2)$. The following preliminary lemma is needed.

\begin{lem}  \label{lem 4} 
Under the assumptions of Lemma ~\ref{prop: H2 est},  we have 
	\begin{eqnarray}
  \nrm{ \frac{p^n + \dt q^n + \phi^n}{2} - \phi^{*,n} }  \le Q_0 \dt^2 ,  \label{lem 4-0} 
	\end{eqnarray}
where $Q_0$ is a constant independent of $\dt$, dependent on the exact solution $\Phi$ and $\varepsilon^{-1}$. 
\end{lem}

\begin{proof} 
A careful calculation reveals that 
\begin{eqnarray} 
  && 
  \frac{p^n + \dt q^n + \phi^n}{2} - \phi^{*,n} 
  = \frac{p^n - \phi^n + \dt q^n}{2} - \frac12 ( \phi^n - \phi^{n-1} )   \nonumber 
\\
  &=&
  \frac12 \dt ( I + \frac34 \varepsilon^2 \dt \Delta^2 )^{-1}  
  \Big( - \varepsilon^2  \Delta^2 ( \frac34 \phi^n + \frac14  \phi^{n-1} )  
    + \Delta F' (\phi^{*,n})  \Big) 
    - \frac12 ( \phi^n - \phi^{n-1} )   \nonumber 
\\
  &=& 
  \frac12 \dt \Big( ( I + \frac34 \varepsilon^2 \dt \Delta^2 )^{-1}  - I \Big) r^n  
  + \frac12 \Big( \dt r^n - (\phi^n - \phi^{n-1} ) \Big) ,     
    \label{lem 4-1-1}  
\\
  &&   \mbox{with} \, \, \, 
  r^n = - \varepsilon^2  \Delta^2 ( \frac34 \phi^n + \frac14  \phi^{n-1} )  
    + \Delta F' (\phi^{*,n})  , \nonumber \label{lem 4-1-2}    
\end{eqnarray} 
in which the representation formulas (\ref{nonlinear-1-2}), (\ref{lem 1-1-1-a}) have been applied. 

For the first part of (\ref{lem 4-1-1}), we see that $\Delta^2 r^n = - \varepsilon^2  \Delta^4 ( \frac34 \phi^n + \frac14  \phi^{n-1} )  + \Delta^3 F' (\phi^{*,n})$, and the following estimates could be derived by \eqref{Hm est-0} : 
\begin{eqnarray} 
  &&
   \| \Delta^4 \phi^n \| , \, \, \| \Delta^4 \phi^{n-1} \| \le C_8 ,   
   \label{lem 4-2-1}    
\\
  && 
  \| \Delta^3 F' (\phi^{*,n}) \| \le \| F' (\phi^{*,n}) \|_{H^6} 
  \le C (  \| \phi^{*,n} \|_{H^6}^3 +  \| \phi^{*,n} \|_{H^6} ) 
  \le C (C_6^3 + C_6) ,
  \label{lem 4-2-2}    
\end{eqnarray}    
in which the 3-D Sobolev embedding and H\"older inequality have been extensively applied, as well as the fact that $F' (\phi) = \phi^3 - \phi$. Then we arrive at 
\begin{eqnarray} 
  \| \Delta^2 r^n \| = \| - \varepsilon^2  \Delta^4 ( \frac34 \phi^n 
  + \frac14  \phi^{n-1} )  + \Delta^3 F' (\phi^{*,n}) \| 
  \le Q_1 := \varepsilon^2 C_8 + C (C_6^3 + C_6)  .  \label{lem 4-2-3}    
\end{eqnarray}   
This in turn leads to 
\begin{eqnarray*} 
    \Big\| \Big( ( I + \frac34 \varepsilon^2 \dt \Delta^2 )^{-1}  - I \Big) r^n \Big\| 
  =  \frac34 \varepsilon^2 \dt 
  \| ( I + \frac34 \varepsilon^2 \dt \Delta^2 )^{-1}  \Delta^2 r^n \|  
  \le \frac34 Q_1 \varepsilon^2 \dt  . \label{lem 4-2-4}   
\end{eqnarray*}   

For the second part of (\ref{lem 4-1-1}), we make use of the following expression of $\phi^n - \phi^{n-1}$ (which comes from the numerical scheme (\ref{scheme-SAVN-1})): 
\begin{align} 
   \phi^n - \phi^{n-1} =& \eta^{n-1/2} \dt \Delta F' (\phi^{*,n-1})  
  - \varepsilon^2  \dt \Delta^2 ( \frac34 \phi^n + \frac14 \phi^{n-2} )  ,  \quad 
  \mbox{so that}  \nonumber 
\\
  \dt r^n - (\phi^n - \phi^{n-1} ) 
  =&  ( \eta^{n-1/2} -1 ) \dt \Delta F' (\phi^{*,n-1})  
   + \dt \Delta ( F' (\phi^{*,n})  -  F' (\phi^{*,n-1}) )   \nonumber 
\\
  &
   + \frac14 \varepsilon^2  \dt \Delta^2 ( \phi^{n-1} - \phi^{n-2} )  .  
   \label{lem 4-3}      
\end{align} 
For the first part of the expansion (\ref{lem 4-3}),  from \eqref{H2 est-0} we have 
\begin{eqnarray}  
  \| \Delta F' (\phi^{*,n-1}) \| \le \| F' (\phi^{*,n-1}) \|_{H^2} 
  \le C (  \| \phi^{*,n-1} \|_{H^2}^3 +  \| \phi^{*,n-1} \|_{H^2} ) 
  \le C (C_2^3 + C_2) .  
    \label{lem 4-4-1}    
\end{eqnarray}    
Its combination with the preliminary rough estimate (\ref{eta est-prelim-2}) implies that 
\begin{eqnarray} 
   \|  ( \eta^{n-1/2} -1 ) \dt \Delta F' (\phi^{*,n-1})  \| 
   \le \frac{C A_{12} \dt^2 }{|S_n^*|}  (C_2^3 + C_2) .  \label{lem 4-4-2}    
\end{eqnarray}   
For the second part of the expansion (\ref{lem 4-3}), we begin with the following identify: 
\begin{eqnarray} 
  &&
   F' (\phi^{*,n})  -  F' (\phi^{*,n-1}) 
   = F'' (\xi^{(3)})(  \phi^{*,n} - \phi^{*,n-1})  ,   \nonumber 
 \\
   && \mbox{with} \, \, \, 
   F'' (\xi^{(3)}) =  (\phi^{*,n})^2  + \phi^{*,n} \phi^{*,n-1} + (\phi^{*,n-1})^2 - 1. 
    \label{lem 4-5-1}    
\end{eqnarray}  
In turn, extensive applications of H\"older inequality and Sobolev inequality indicate that 
\begin{eqnarray} 
  \| \Delta ( F' (\phi^{*,n})  -  F' (\phi^{*,n-1}) )  \| 
  &\le& C \|  F'' (\xi^{(3)}) \|_{H^2}^2 \| \phi^{*,n} - \phi^{*,n-1} \|_{H^2}  \nonumber 
\\
  &\le& 
  C \|  F'' (\xi^{(3)}) \|_{H^2}^2 ( \| \phi^n - \phi^{n-1} \|_{H^2}  
  +  \| \phi^{n-1} - \phi^{n-2} \|_{H^2} )    \nonumber 
\\
  &\le& 
  C (C_2^2 +1) \cdot D_2 \dt ,     
  \label{lem 4-5-2}    
\end{eqnarray}    
in which the $H^2$ estimate (\ref{H2 est-0}) and the discrete temporal derivative estimate (\ref{time-stability-1-bound}) have been applied in the last step. For the third part of the expansion (\ref{lem 4-3}), we make use of (\ref{time-stability-1-bound}) again, and get 
\begin{eqnarray} 
   \| \frac14 \varepsilon^2  \dt \Delta^2 ( \phi^{n-1} - \phi^{n-2} )  \| 
   \le \frac14 \varepsilon^2  \dt  \cdot D_4 \dt = \frac14 D_4 \varepsilon^2  \dt^2 . 
   \label{lem 4-6}    
\end{eqnarray} 
Therefore, a substitution of (\ref{lem 4-4-2}), (\ref{lem 4-5-2}) and (\ref{lem 4-6}) into (\ref{lem 4-3}) leads to the following estimate: 
\begin{equation} 
  \|  \dt r^n - (\phi^n - \phi^{n-1} )  \|  \le Q_2 \dt^2 ,  \quad \mbox{with} \, \, \, 
  Q_2 = \frac{C A_{12}}{|S_n^*|}  (C_2^3 + C_2) 
  + C (C_2^2 +1) D_2   + \frac14 D_4 \varepsilon^2 .  \label{lem 4-7}    
\end{equation}  

  Finally, a combination of (\ref{lem 4-1-1}), (\ref{lem 4-2-2}) and (\ref{lem 4-7}) yields the desired result (\ref{lem 4-0}), by taking $Q_0 = \frac38 Q_1 \varepsilon^2  + \frac12 Q_2$. This finishes the proof of Lemma~\ref{lem 4}.        
\end{proof} 

As a consequence, we are able to obtain a sharper estimate for $g(1)$, in comparison with Lemma~\ref{lem 2}. 

\begin{lem}  \label{lem 5} 
Under the assumptions of Lemma ~\ref{prop: H2 est},  we have $| g_n (1) | \le A_{18} \dt^3$, with $Q_3$ a constant independent of $\dt$, but dependent on the exact solution $\Phi$. 
\end{lem} 

\begin{proof} 
  The expansion formula (\ref{lem 2-1}) of $g(1)$ indicates that 
\begin{align} 
  g_n (1) = J_1 + J_2 ,    \quad 
   & J_1 = \int_\Omega \, \Big( F (p^n + \dt q^n) - F (\phi^n)   
  - F' ( \frac{p^n + \dt q^n + \phi^n}{2} ) (p^n + \dt q^n - \phi^n) \Big) \, d \bx  ,  
  \nonumber 
\\
  & 
  J_2 =  
  \int_\Omega \, \Big( ( F' ( \frac{p^n + \dt q^n + \phi^n}{2} )  
  - F' (\phi^{*,n}) ) (p^n + \dt q^n - \phi^n) \Big) \, d \bx .  
  \label{lem 5-1}   
\end{align} 
For the $J_1$ part, we make use of the following identity 
\begin{equation*} 
\begin{aligned}  
  & 
   F (x) - F (y) - F' (\frac{x+y}{2}) (x-y) 
   = \frac{x+y}{8} ( x-y)^3 ,  \quad \forall x, y \in \mathrm{R} ,  
   \quad \mbox{so that}    
\\
  &  
   F (p^n + \dt q^n) - F (\phi^n)   
  - F' ( \frac{p^n + \dt q^n + \phi^n}{2} ) (p^n + \dt q^n - \phi^n)   
\\
  &
  = \frac18 (p^n + \dt q^n + \phi^n)  \cdot ( p^n - \phi^n + \dt q^n )^3 .  
\end{aligned} 
    \label{lem 5-2-2}      
\end{equation*}  
In turn, with an application of H\"older inequality, combined with the inequalities~(\ref{lem 1-0-1}), (\ref{lem 2-4}), we obtain 
\begin{eqnarray} 
  |   J_1 | 
  &\le& \frac18 \| p^n + \dt q^n + \phi^n \|_{L^\infty} \cdot 
  \| p^n - \phi^n + \dt q^n \|_{L^\infty}^3 \cdot | \Omega | \nonumber 
\\
  &\le& 
   \frac18 (A_2 + C C_2 + 1) (A_1 + A_3)^3 | \Omega | \dt^3 .   
  \label{lem 5-2-3}      
\end{eqnarray}   
For the $J_2$ part, we begin with the following application of intermediate value theorem: 
\begin{eqnarray*} 
   F' ( \frac{p^n + \dt q^n + \phi^n}{2} ) - F' (\phi^{*,n})  
   = F'' (\xi^{(4)}) \cdot ( \frac{p^n + \dt q^n + \phi^n}{2} - \phi^{*,n} ) ,   
   \label{lem 5-3-1}      
\end{eqnarray*}      
with $\xi^{(4)}$ between $\frac{p^n + \dt q^n + \phi^n}{2}$ and $\phi^{*,n}$. By the fact that $F'' (\phi) = 3 \phi^2 -1$ and the $L^\infty$ bounds for $\frac{p^n + \dt q^n + \phi^n}{2}$ and $\phi^{*,n}$, we get 
\begin{eqnarray*} 
  \| F'' (\xi^{(4)}) \|_{L^\infty}  \le C \Big( \| \frac{p^n + \dt q^n + \phi^n}{2} \|_{L^\infty}^2 
  + \| \phi^{*,n} \|_{L^\infty}^2 + 1 \Bigr) 
  \le C (A_2^2 + C_2^2 + 1) .  \label{lem 5-3-2}      
\end{eqnarray*}      
Then we arrive at 
\begin{eqnarray} 
  | J_2 | &\le& \| F'' (\xi^{(4)}) \|_{L^\infty}   
  \cdot \nrm{ \frac{p^n + \dt q^n + \phi^n}{2} - \phi^{*,n} }  
  \cdot  \| p^n + \dt q^n - \phi^n \|_{L^\infty} \cdot | \Omega |^{1/2}  \nonumber 
\\
  &\le& C (A_2^2 + C_2^2 + 1)   
  \cdot Q_0 \dt^2 \cdot ( A_1 + A_3 ) \dt  \cdot | \Omega |^{1/2}  \nonumber  
\\
  &=& C Q_0 | \Omega| (A_2^2 + C_2^2 + 1)  
  ( A_1 + A_3 ) \dt^3 ,   \label{lem 5-3-3}  
\end{eqnarray} 
in which the estimate (\ref{lem 4-0}) (in Lemma~\ref{lem 4}) has been applied in the second step. Therefore, a combination of (\ref{lem 5-1}), (\ref{lem 5-2-3}) and (\ref{lem 5-3-3}) yields the desired result, by taking $A_{18} = \frac18 (A_2 + C C_2 + 1) (A_1 + A_3)^3 | \Omega | + C Q_0 | \Omega| (A_2^2 + C_2^2 + 1)  ( A_1 + A_3 )$. This finishes the proof of Lemma~\ref{lem 5}. 
\end{proof} 

Next, we have to obtain a refined estimate for $| \eta^{n+1/2} -1|$. 
  
\begin{lem}  \label{thm: eta est} 
Under the assumptions of Lemma ~\ref{prop: H2 est},  we have $| \eta^{n+1/2} -1 | \le \frac{2 A_{18}}{\gamma} \dt^2$ with $A_{18}$ dependent on the exact solution $\Phi$. 
\end{lem}

\begin{proof} 
  Again, it is assumed that  $g_n (1) > 0$, without loss of generality. Following the proof of Theorem~\ref{thm: solvability}, we see that 
\begin{eqnarray*} 
   0 \le g_n (1) \le A_{18} \dt^3 ,  \quad 
   g'_	n (\eta) \ge \frac{\gamma}{2} \dt ,  \, \, \, 
   \mbox{for $1 - \dt^{1/2} \le \eta \le 1 + \dt^{1/2}$}  . 
   \label{eta est-1} 
\end{eqnarray*} 
Then we conclude that, $g_n (\eta)$ is monotone and increasing over the interval $(1 - \delta^{**}, 1)$, with $\delta^{**} =  \frac{2 A_{18}}{\gamma} \dt^2$, so that 
\begin{eqnarray*} 
  g_n(1) \ge 0 , \, \, \, g_n (1 - \delta^{**}) \le g_n(1) - \frac{\gamma}{2} \dt \cdot \delta^{**} 
   \le A_{18} \dt^3  -   A_{18} \dt^3 = 0 .  \label{eta est-2} 
\end{eqnarray*}  
Therefore, there is a unique solution for $g_n (\eta) =0$ over the interval $(1 - \delta^{**}, 1)$. This finishes the proof of Theorem~\ref{thm: eta est}.    
\end{proof}

\subsection{Proof of Theorem 4.1}

First of all, it is clear that both the exact solution $\Phi(\cdot,t)$ and the numerical solution $\{\phi^k\}$ are mass conservative: 
\begin{eqnarray} 
   \overline{\Phi^k} = \overline{\Phi^0} ,  \quad  \overline{\phi^k} = \overline{\phi^0} ,  
   \, \, \, \forall k \ge 0 ,    
   \label{mass conserv-1} 
\end{eqnarray} 
in which $\Phi^k$ is the exact solution evaluated at the time instant $t^k$. Then we conclude that the numerical error function must have zero-average at each time step: 
\begin{eqnarray} 
  \overline{e^k} = 0 ,  \, \, \, \forall k \ge 0 ,  \quad 
  \mbox{by taking $\phi^0 = \Phi^0$} . 
\end{eqnarray} 
In turn, the following elliptic regularity estimates are available: 
\begin{eqnarray} 
  \| e^k \|_{H^{2m}} \le C \| \Delta^{m} e^k \|  ,  \quad \mbox{for $m = 1, 2, 3$}, 
  \quad \forall k \ge 0.  \label{mass conserv-3} 
\end{eqnarray}

First, the {\it a priori} assumption \eqref{a priori-1} is made. As a result, the unique solvability for $\eta^{n+1/2}$ (in the range of $1 - \dt^{1/2} \le \eta^{n+1/2} \le 1 + \dt^{1/2}$) is assured by Theorem~\ref{thm: solvability}, and a refined estimate in Lemma~\ref{thm: eta est} becomes available. 

The following truncation error analysis for the exact solution can be obtained by using a straightforward Taylor expansion: 
\begin{equation}  
  \frac{\Phi^{n+1} - \Phi^n}{\dt} = \Delta \Bigl( F' (\Phi^{*,n}) 
  - \varepsilon^2  \Delta ( \frac34 \Phi^{n+1} + \frac14 \Phi^{n-1} ) \Bigr) 
  + \tau^n ,   \quad \Phi^{*,n} = \frac32 \Phi^n - \frac12 \Phi^{n-1} , 
   \label{consistency-1}   
\end{equation} 
with $\| \Delta \tau^n \| \le C \dt^2$. 

In turn, subtracting the numerical scheme~(\ref{scheme-SAVN-1}) from the consistency estimate~(\ref{consistency-1}) yields
\begin{eqnarray}  
  \frac{e^{n+1} - e^n}{\dt} &=& (1 - \eta^{n+1/2} ) \Delta  F' (\Phi^{*,n})  \nonumber 
\\
  && 
  + \Delta \Bigl( \eta^{n+1/2} ( F' (\Phi^{*,n}) - F' (\phi^{*,n})  ) 
  - \varepsilon^2  \Delta ( \frac34 e^{n+1} + \frac14 e^{n-1} ) \Bigr) 
  + \tau^n .    
   \label{error equation-1}   
\end{eqnarray} 
Taking an inner product with (\ref{error equation-1}) by $2 \Delta^4 e^{n+1}$ gives 
\begin{equation}  
\begin{aligned} 
   &
  \| \Delta^2 e^{n+1} \|^2 - \| \Delta^2 e^n \|^2 +  \| \Delta^2 ( e^{n+1} - e^n ) \|^2  
  +  \varepsilon^2  \dt \Bigl( \Delta^3 ( \frac32 e^{n+1} + \frac12 e^{n-1} ) , 
    \Delta^3 e^{n+1} \Bigr)  
\\
  =& 2 (1 - \eta^{n+1/2} ) \dt ( \Delta^2  F' (\Phi^{*,n}) , \Delta^3 e^{n+1} ) 
  + 2 \eta^{n+1/2} \dt ( \Delta^2 ( F' (\Phi^{*,n}) - F' (\phi^{*,n}) ) , 
  \Delta^3 e^{n+1} )  
\\
  &
  + 2 \dt ( \Delta \tau^n ,  \Delta^3 e^{n+1} ) ,  
\end{aligned} 
   \label{convergence-1}   
\end{equation} 
in which the integration by parts have been extensively applied in the derivation. The surface diffusion term could be analyzed similarly  as in~(\ref{H2 est-2}): 
	\begin{eqnarray}
  &&(  \Delta^3 e^{n+1}  ,  \Delta^3 ( \frac32 e^{n+1} + \frac12 e^{n-1} ) ) 
   =  \frac32  \| \Delta^3 e^{n+1} \|^2   
   + \frac12  (  \Delta^3 e^{n+1}  , \Delta^3 e^{n-1} )
 	\nonumber  
	\\ 
&&\ge  \frac32  \|  \Delta^3 e^{n+1} \|^2  
- \frac12  \Big( \frac12  \| \Delta^3 e^{n+1} \|^2 
 +\frac12 \| \Delta^3 e^{n-1} \|^2  \Big)
	\nonumber 
	\ge  \frac54  \|  \Delta^3 e^{n+1} \|^2 -  \frac14 \| \Delta^3 e^{n-1} \|^2 .
	\label{convergence-2}    
	\end{eqnarray}
The local truncation error term could be bounded by the Cauchy inequality: 
\begin{eqnarray} 
  2 ( \Delta \tau^n ,  \Delta^3 e^{n+1} )  
  \le 2 \| \Delta \tau^n \| \cdot \|  \Delta^3 e^{n+1} \|  
  \le  \varepsilon^{-2} \| \Delta \tau^n \|^2 
  + \frac14 \varepsilon^2 \|  \Delta^3 e^{n+1} \|^2 . 
    \label{convergence-3}    
	\end{eqnarray}      

For the first nonlinear inner product, we recall the refined estimate, $| \eta^{n+1/2} -1 | \le Q_3 \dt^2$, as established in Lemma~\ref{thm: eta est}. In addition, the following estimate could be derived: 
\begin{eqnarray} 
   \Delta^2  F' (\Phi^{*,n}) 
   \le C ( \| \Phi^{*,n} \|_{H^4}^3 + \| \Phi^{*,n} \|_{H^4} ) 
   \le C ( (A^*)^3 + A^* ) , 
   \label{convergence-4-1}    
	\end{eqnarray}   
based on the fact that $F' (\phi) = \phi^3 - \phi$, and the bound~(\ref{exact-bound-1}) for the exact solution. Then we arrive at 
\begin{equation} 
\begin{aligned} 
  & 2 (1 - \eta^{n+1/2} )  ( \Delta^2  F' (\Phi^{*,n}) , \Delta^3 e^{n+1} )  
   \le 3  Q_3 \dt^2 \cdot \| \Delta^2  F' (\Phi^{*,n}) \| 
     \cdot \| \Delta^3 e^{n+1} \|   
\\
  &\le 
      C   ( (A^*)^3 + A^* ) Q_3 \dt^2  \| \Delta^3 e^{n+1} \|   
  \le
   C   ( (A^*)^3 + A^* )^2 Q_3^2 \varepsilon^{-2} \dt^4   
  + \frac14 \varepsilon^2 \|  \Delta^3 e^{n+1} \|^2 .  
\end{aligned} 
  \label{convergence-4-2}     
\end{equation}  
For the other nonlinear error term, we begin with the following application of intermediate value theorem: 
\begin{equation*} 
\begin{aligned}  
     F' (\Phi^{*,n}) - F' (\phi^{*,n}) 
     = F'' (\xi^{(5)}) e^{*,n},  \, \, \,  \mbox{with}  & \, \, \, 
      F'' (\xi^{(5)}) = (\Phi^{*,n})^2 + \Phi^{*,n} \phi^{*,n} + (\phi^{*,n})^2 -1 ,  
\\
   &  \, \, \, 
      e^{*,n} = \frac32 e^n - \frac12 e^{n-1} . 
\end{aligned} 
  \label{convergence-4-3}        
\end{equation*} 
This in turn leads to the following estimate: 
\begin{equation*} 
\begin{aligned} 
   \|  \Delta^2 ( F' (\Phi^{*,n}) - F' (\phi^{*,n}) )   \|  
    =& \|  \Delta^2 ( F'' (\xi^{(5)}) e^{*,n} )   \|  
   \le C ( \|  \Phi^{*,n} \|_{H^4}^2 + \|  \phi^{*,n} \|_{H^4}^2 + 1) 
   \cdot \|  e^{*,n} \|_{H^4}  
\\
   \le& 
   C ( (A^*)^2 + C_4^2 + 1) \cdot \|  \Delta^2 e^{*,n} \|  .   
\end{aligned} 
  \label{convergence-4-4}        
\end{equation*} 
Consequently, the following inequality is available: 
\begin{eqnarray} 
  && 
   2 \eta^{n+1/2} ( \Delta^2 ( F' (\Phi^{*,n}) - F' (\phi^{*,n}) ) , 
  \Delta^3 e^{n+1} )   
  \le 
    C ( (A^*)^2 + C_4^2 + 1)  \|  \Delta^2 e^{*,n} \|    
  \cdot \| \Delta^3 e^{n+1} \|    \nonumber  
\\
  &\le& 
    C ( (A^*)^2 + C_4^2 + 1)^2 \varepsilon^{-2}  \|  \Delta^2 e^{*,n} \|^2   
  + \frac14 \varepsilon^2 \| \Delta^3 e^{n+1} \|^2  \nonumber 
\\
  &\le& 
  C ( (A^*)^2 + C_4^2 + 1)^2 \varepsilon^{-2} 
  ( \|  \Delta^2 e^n \|^2  + \|  \Delta^2 e^{n-1} \|^2  ) 
  + \frac14 \varepsilon^2 \| \Delta^3 e^{n+1} \|^2  .
  \label{convergence-4-5}        
\end{eqnarray} 

Subsequently, a substitution of (\ref{convergence-2}), (\ref{convergence-3}), (\ref{convergence-4-2}) and (\ref{convergence-4-5}) into (\ref{convergence-1}) leads to 
\begin{equation*} 
\begin{aligned}   
  &
  \| \Delta^2 e^{n+1} \|^2 - \| \Delta^2 e^n \|^2 
  + \frac12 \varepsilon^2 \dt \|  \Delta^3 e^{n+1} \|^2 
  -  \frac14 \varepsilon^2 \dt \| \Delta^3 e^{n-1} \|^2  
\\
  \le& 
   C   ( (A^*)^3 + A^* )^2 Q_3^2 \varepsilon^{-2} \dt^5 
  +  Q_4 \varepsilon^{-2} 
  ( \|  \Delta^2 e^n \|^2  + \|  \Delta^2 e^{n-1} \|^2  )   
  + \varepsilon^{-2} \dt \| \Delta \tau^n \|^2 ,    
\end{aligned} 
   \label{convergence-5}   
\end{equation*} 
with $Q_4 = C ( (A^*)^2 + C_4^2 + 1)^2$. Therefore, an application of discrete Gronwall inequality yields the desired convergence estimate~(\ref{convergence-0}). 

Finally, we have to recover the   {\it a-priori}  assumption~(\ref{a priori-1}) at time step $t^{n+1}$ to close the induction argument. By the convergence estimate (\ref{convergence-0}), we get 
	\begin{equation*}
\| \Delta^2 e^{n+1} \| \le \hat{C} \dt^2 ,  \quad 
 \varepsilon^2 \dt \| \Delta^3 e^{n+1} \|^2  \le \hat{C}^2 \dt^4 , \, \, \, 
 \mbox{so that} \, \,  \| \Delta^3 e^{n+1} \| \le \hat{C} \varepsilon^{-1} \dt^{3/2} . 
  \label{a priori-4}
	\end{equation*}
Making use of the elliptic regularity estimate (\ref{mass conserv-3}), we have 
\begin{eqnarray*} 
  \| e^{n+1} \|_{H^4} \le C \| \Delta^2 e^{n+1} \|  \le Q_5 \hat{C} \dt^2 ,  \quad 
  \| e^{n+1} \|_{H^6} \le C \| \Delta^3 e^{n+1} \|  
  \le Q_6 \hat{C} \varepsilon^{-1} \dt^{3/2} .  \label{a priori-5} 
\end{eqnarray*} 
Therefore, the   {\it a-priori}  assumption~(\ref{a priori-1}) is valid at time step $t^{n+1}$, provided that 
\begin{equation*} 
  \dt \le \min \Big( ( Q_5 \hat{C} )^{-2} , Q_6^2 \hat{C}^2 \varepsilon^2,\frac{\gamma^4}{4} \Big) . 
  \label{a priori-6} 
\end{equation*}        
This finishes the complete induction argument for both the solvability analysis and convergence estimate. 

\section{Concluding remarks} \label{sec:conclusion} 
	The unique solvability and error analysis of the original Lagrange multiplier approach, proposed in \cite{ChengQ2020a} for gradient flows, is considered in this article. We first identified a necessary and sufficient  condition that  must be satisfied for the system of nonlinear algebraic equations, arising from  the original  Lagrange multiplier approach,  to admit a unique solution in the neighborhood of its exact solution. Afterward, we proposed a modified Lagrange multiplier approach to ensure  that the computation can continue even if the aforementioned condition was not satisfied. 
	
	However, a globally unique solvability of the nonlinear system of algebraic equations is not possible, due to the non-monotone feature of the implicit part in the numerical system.  Using the Cahn-Hilliard equation and a second-order modified Crank-Nicholson scheme with a Lagrange multiplier  as an example, we have provided a unique solvability analysis using  localized estimates in the neighborhood of its exact solution under  an   {\it a-priori}  assumption, and then carried out a rigorous error analysis, which in turn recovered the  {\it a-priori}  assumption using an inductive argument. 
	
	The results presented in this paper are  the first  unique solvability analysis and error estimate for the original  Lagrange multiplier approach. In a future work, we shall consider the theoretical analysis for a more general  Lagrange multiplier approach for gradient flows with global constraints proposed in \cite{ChengQ2020b}. Due to the multiple Lagrange multipliers involved in the latter approach, the analysis would be much more challenging, while the analytical techniques introduced in this paper may be helpful.

	\section*{Acknowledgements}
This work is supported in part by NSFC 12371409 (J. Shen), NSFC 12301522 (Q. Cheng) and  NSF DMS-2012269, DMS-2309548 (C.~Wang).

 \bibliographystyle{plain}
\bibliography{draft3}
\end{document}